\documentclass[letter, 12pt]{amsart}
\usepackage[dvipsnames,usenames]{color}

\newcommand{\floor}[1] {\left\lfloor #1 \right\rfloor}

\usepackage{extsizes}
\usepackage{blindtext}
\usepackage[utf8]{inputenc}
\usepackage{amsmath}
\usepackage{amssymb}
\usepackage{amsthm}
\usepackage{pst-node}
\usepackage{blindtext}
\usepackage{geometry}
 \usepackage{float}
 \usepackage{amssymb}
\usepackage{tikz-cd}
\usepackage{amscd}
\usepackage[all,color]{xy}
\usepackage{sseq}
\usepackage{amsfonts}
\usepackage{enumerate}
\usepackage{graphicx}
\usepackage{fancyhdr}
\usepackage{tikz}
\usetikzlibrary{cd}
\usetikzlibrary{arrows}
\usetikzlibrary{decorations.pathreplacing}
\usepackage{hyperref}
\usepackage{mathtools}
\usepackage{bm}
\usepackage{subfig}
\theoremstyle{plain} \numberwithin{equation}{section}

\newtheorem{theorem}{Theorem}[section]

\newtheorem{corollary}[theorem]{Corollary}

\newtheorem{lemma}[theorem]{Lemma}

\newtheorem{proposition}[theorem]{Proposition}
\theoremstyle{definition}

\newtheorem{remark}[theorem]{Remark}

\newcommand{\x}{\times}

\newcommand{\s}{$\mathfrak{s}$}
\def \f-{f^{-1}}
\def \fp-{f^{-1}_\partial}

\usepackage{hyperref}

\def\Z{\mathbb{Z}}
\def\Q{\mathbb{Q}}
\def\F{\mathbb{F}}

\def\II {\mathcal{I}}

\def\JJ {\mathcal{J}}

\def\l {\ell}

\def\x {\mathbf{x}}
\DeclareMathOperator{\gr}{gr}

\def\CF {\operatorname{CF}}
\def\HF {\operatorname{HF}}
\def\CFKa {\widehat{\operatorname{CFK}}}
\def\HFKa {\widehat{\operatorname{HFK}}}

\def\CFp {\operatorname{CF}^+}
\def\CFm {\operatorname{CF}^-}
\def\CFi {\operatorname{CF}^\infty}
\def\CFa {\operatorname{\widehat{CF}}}

\def\HFa {\operatorname{\widehat{HF}}}

\def\CFK {\operatorname{CFK}}

\def\CFKm {\operatorname{CFK}^-}
\def\HFKm {\operatorname{HFK}^-}
\def\CFKi {\operatorname{CFK}^\infty}

\DeclareMathOperator{\Spin}{Spin}
\DeclareMathOperator{\tb}{tb}
\DeclareMathOperator{\rot}{rot}
\def\spincs {\mathfrak{s}}
\def\X {\mathcal{X}}
\def\Xa {\widehat{\X}}
\def\d {\partial}

\newif\ifrevision
\newcommand{\edit}[1]{%
\ifrevision
\textcolor{blue}{#1}%
\else
#1%
\fi
}

\title{Negative contact surgery on Legendrian non-simple knots}
\author{Shunyu Wan}
\address{Department of Mathematics, Georgia     Institute of Technology, Atlanta, GA, USA}
\email{swan48@gatech.edu}

\author{Hugo Zhou}
\address{Department of Mathematics, University of Michigan, Ann Arbor, MI, USA}
\email{hugozhou@umich.edu}

\begin{document}
\revisionfalse
\maketitle
\begin{abstract}
    We prove that for any pair of Legendrian representatives of the Chekanov-Eliashberg twist knots with different LOSS invariants, any negative rational contact $r$-surgery with $r\neq -1$ always gives rise to different contact 3-manifolds distinguished by their contact invariants.
This gives the first examples of pairs of Legendrian knots with the same classical invariants but distinct contact $r$-surgeries for all negative rational number $r$. 
  We also generalize the statement from the twist knots to a certain families of \edit{two-bridge} knots. 
\end{abstract}
\section{Introduction}
In \cite{EtnyreOncontactsurgery} Etnyre first asked the question \edit{of} whether Legendrian surgery, i.e. contact $(-1)$-surgery, on distinct Legendrian knots in the standard tight
contact structure on $S^3$ always produces distinct contact manifolds, and especially whether this is the case for the Chekanov-Eliashberg twist knots $E_n$ (Figure \ref{subfig:En}). Later using linearized contact homology Bourgeois-Ekholm-Eliashberg showed that Legendrian surgery on max-$\tb$, non Legendrian isotopic representatives of the twist knots $E_n$ \edit{gives} different contact 3-manifolds  \cite{BourgeoisEkholmEliashbergEffectofLegendriansurgery}. However, it is not known whether Legendrian surgery on the stabilized Legendrian twist knots (or equivalently contact negative integer surgery on non-stabilized twist knots) gives different contact 3-manifolds or not. The Bourgeois-Ekholm-Eliashberg argument does not directly apply in those cases, since the Legendrian DGA vanishes for stabilized Legendrians \cite{ChekanovDifferentialAlgebraOfLegendrianLinks}.  

On the other hand, one can consider invariants from Heegaard Floer theory, namely the contact invariant and the LOSS invariant, for contact 3-manifolds and Legendrian knots respectively. Since the LOSS invariant is unchanged under negative stabilization of the Legendrian knot \cite{LOSS}, it makes the calculation of Legendrian surgery on \edit{negatively stabilized} knot possible. 

In \cite{OzsvathStipsiczContactsurgeryandtransverseinvariant} Ozsv\'ath and Stipsicz showed that for the Eliashberg–Chekanov twist knot $E_n$ with $n>3$ and odd, there are $\lceil \frac{n}{4} \rceil$ different Legendrian representatives of $E_n$ in $(S^3, \xi_{std})$ with \edit{Thurston--Bennequin number $1$ (which is maximal)} and rotation number $0$ that have different LOSS invariants. Moreover the classification of Legendrian and transverse \edit{twist} knots by Etnyre-Ng-V\'{e}rtesi \cite{EtnyreNgVertesiLegendrianandtransversetwistknots} implies that the ones Ozsv\'ath and Stipsicz found are all the Legendrian representatives of $E_n$ that could have different LOSS invariant. Those twist knots will be the key objects for this paper. 
\begin{figure}[htb!]
\centering
      \begin{minipage}{0.5\linewidth}
\centering
\subfloat[]{
\begin{tikzpicture}
    \node at (-5,0){\includegraphics[scale=0.35]{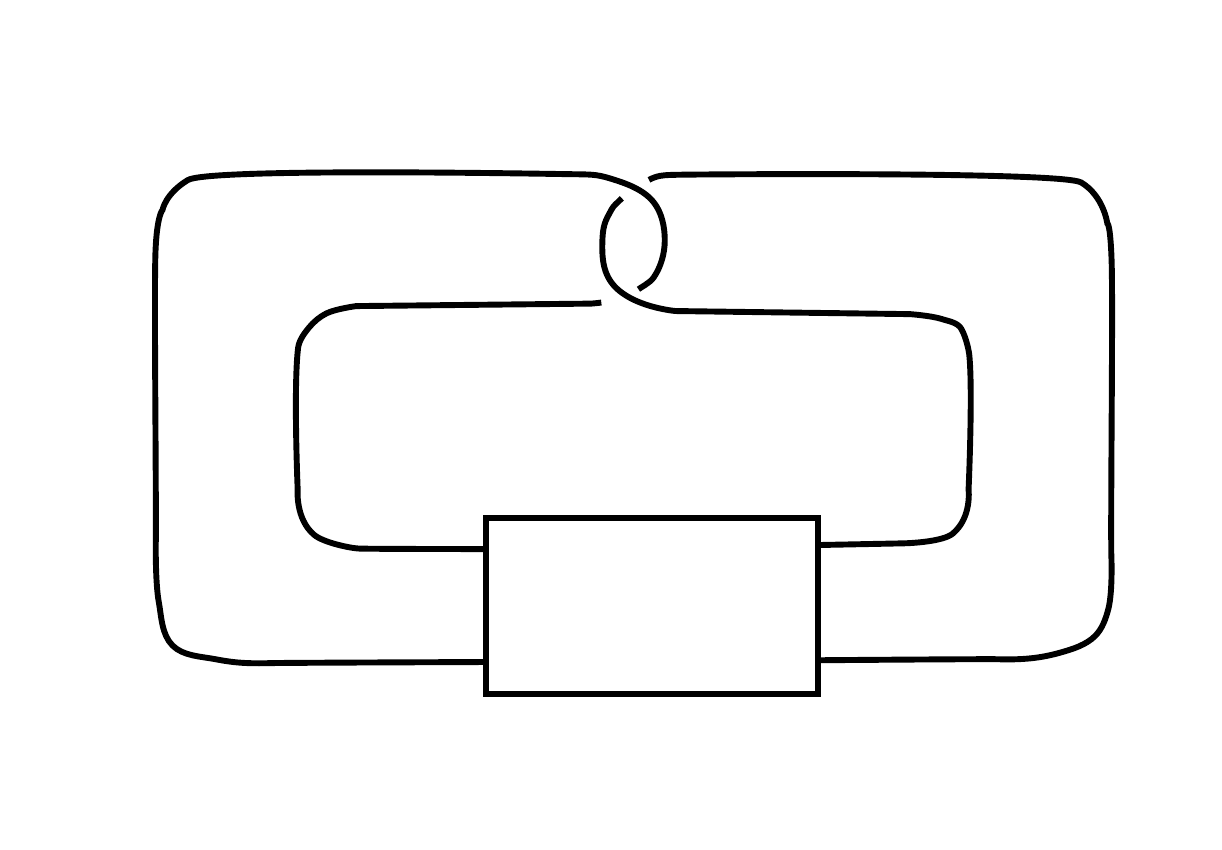}};
      \node at (-4.8,-1.1){ $-n-1$};
      \end{tikzpicture}
      \label{subfig:En}
}
\end{minipage}%
\begin{minipage}{0.5\linewidth}
\centering
\subfloat[]{
\begin{tikzpicture}
        \node at (1.2,-1.1){ $-n$};
           \node at (-0.9,-1.1){ $m$};
       \node at (0,0){\includegraphics[scale=0.35]{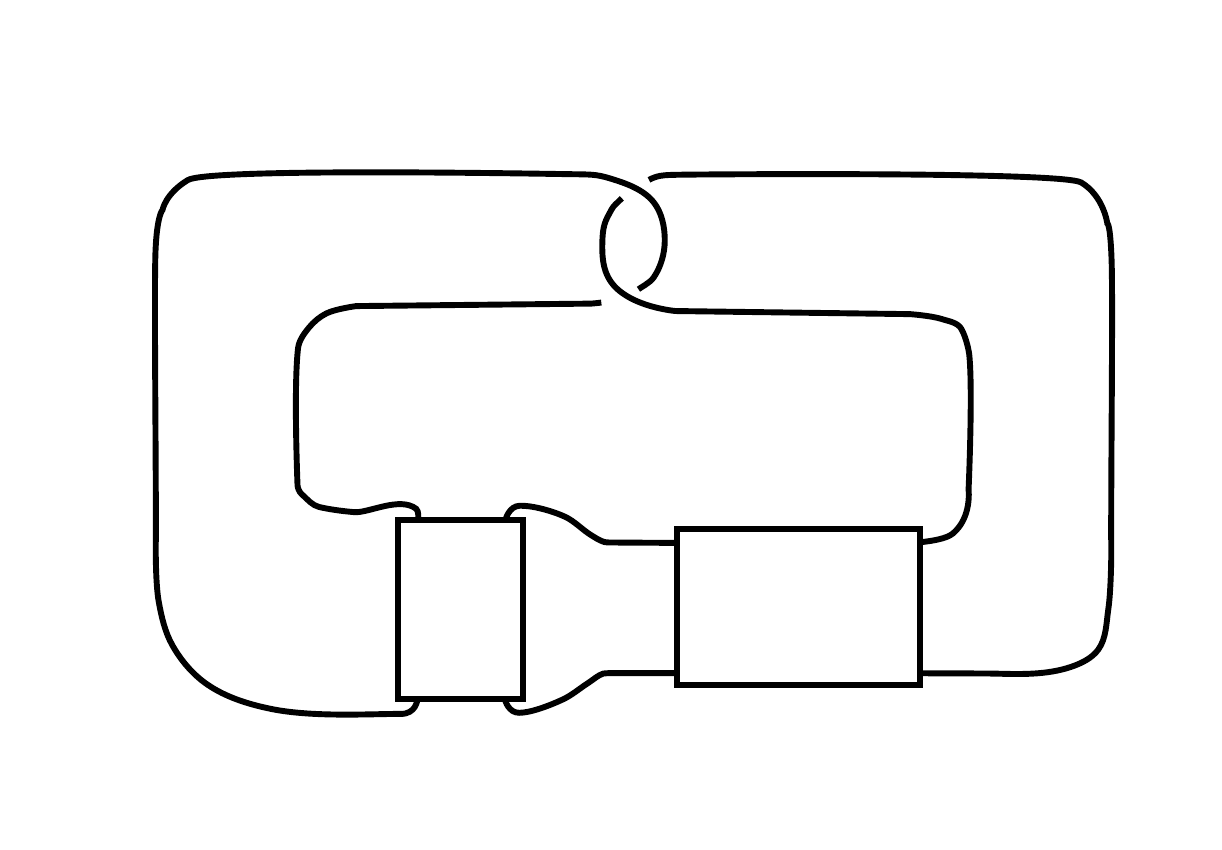}};
\end{tikzpicture}
      \label{subfig:Emn}
}
\end{minipage}%
    \caption{The numbers in boxes indicate half-twists (positive for the right-handed twists and negative for the left-handed twists). The diagram on the left is the Eliashberg–Chekanov twist knot $E_n$. The diagram on the right depicts the knot $E(m,n)$, where taking $m=1$ recovers $E_n$.   }
    \label{fig:E}
\end{figure}
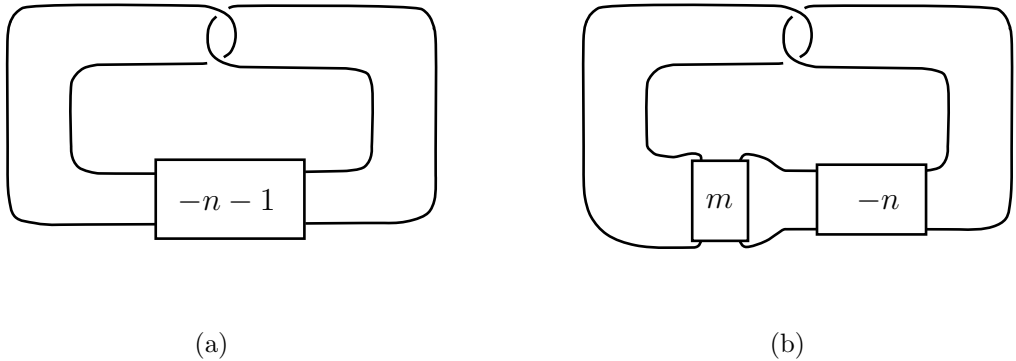 

\edit{Recall that in order to define a contact structure on the contact $r$-surgery one needs to make a choice of stabilizations.}
Throughout the paper we only consider \edit{the} contact structure that corresponds to stabilizations being all negative (see Section \ref{general contact surgery} for more details). Under the above convention we show the following.  
\begin{theorem} \label{negative rational contact surgery on twist knots}
    Let $L_1$ and $L_2$ be two Legendrian representatives of $E_n$ in $(S^3, \xi_{std})$ with $n>3$ odd that have same $\tb=1$ and $\rot=0$ but different LOSS invariants. 
Then for any negative rational number $r \neq -1$,  \edit{contact $r$-surgery} on $L_1$ and $L_2$ \edit{results} in contact manifolds with different Heegaard Floer contact invariants. 
\end{theorem}

The main tools used in proving 
Theorem \ref{negative rational contact surgery on twist knots} are the naturality of contact/LOSS invariant and explicit mapping cone calculations. We first prove the case when $r=-2$ by showing that the LOSS invariant of the induced Legendrian knots of $L_1$ and $L_2$ are different, then we show the map from the LOSS invariant to the contact invariant is injective. The general case is proved by using the result when $r=-2$ and  \edit{an algorithm by Ding-Geiges-Stipsicz \cite{DGS} (which we will  refer to as the DGS algorithm henceforth).}


It is interesting to point out that when $r=-1$, Lisca and Stipsicz \cite{LiscaStipsiczNotesoncontactinvariants} showed that for any max $\tb$ representative of $E_n,$ regardless of the LOSS invariant, the contact $(-1)$-surgery always results in manifolds with the same contact invariant. \edit{On the other hand}, Bourgeois-Ekholm-Eliashberg showed that Legendrian \edit{surgery} on  the max $\tb$ representatives of $E_n$ \edit{yields} different contact 3-manifolds \cite{BourgeoisEkholmEliashbergEffectofLegendriansurgery}. Combining with our result, we have the following.

\begin{corollary}
There exist arbitrarily large (finite) families of Legendrian knots in $(S^3,\xi_{std})$ such that knots in each family have the same smooth knot type, $\tb$, and $\rot$, but different contact $r$-surgery (up to contact isotopy) for any given rational number $r<0$. 
\end{corollary}

When we take the $4$-dimensional perspective  of Legendrian surgery, i.e. as the attachment of Stein handles, then the above corollary immediately implies the following. 

\begin{corollary}
    For any integer $N\geq 2$, there exist infinitely many examples of smooth 4-manifolds $X$  \edit{each admitting} $N$ inequivalent Stein structures, all inducing isomorphic $\Spin^c$ structures on $X$. The Stein structures are distinguished by the contact invariants on the contact boundary.
\end{corollary}

\begin{remark}
    For $N\geq 2$, the 4-manifolds are obtained by attaching two handles along $E_{4N-3}$ (or $E_{4N-1}$) with any smooth framing less than $0$. In \cite{KarakurtObaUkidaPlanarLFandStein} Karakurt-Oba-Ukida \edit{give} infinitely many examples of contractible manifolds each admitting $N=2$ inequivalent Stein structures with the same $\Spin^c$ structure, but do not address the case $N\geq 3$. On the other hand the results from Bourgeois-Ekholm-Eliashberg \cite{BourgeoisEkholmEliashbergEffectofLegendriansurgery} yield examples for all $N\geq 2$, but \edit{no} infinite families (their examples cannot be distinguished by the contact invariants).
\end{remark}

It is also worth mentioning that for positive integer contact surgery the resulting contact invariant is determined by the classical invariants of the Legendrian knots and the underlying contact manifolds \cite[Corollary 1.6]{WanNaturalityofLOSSinvariant} (see also \cite{MarcoOSinvariantsofcontactsurgeries}), and it is natural to ask whether the same is true for negative surgery. Theorem \ref{negative rational contact surgery on twist knots} gives negative answer to this question. 

\begin{corollary}
  The contact invariant of negative contact surgery is not necessarily determined by the classical invariants of the Legendrian knot.  
\end{corollary}

Using the naturality of LOSS invariant and the Legendrian representatives of $E_n$ the first author found families of \edit{two-bridge} knots with the same $\tb$ and $\rot$ but different LOSS invariants:
\begin{theorem}[{\cite[Theorem 5.2]{WanNaturalityofLOSSinvariant}}] \label{Legendrian two bridge knots} 
    Let $m,n$ be positive odd integers with $n>3$. The knot  $E(m,n)$ (Figure \ref{subfig:Emn}) has at least $\lceil \frac{n}{4} \rceil$ Legendrian representatives that have  $\tb=m$ and $\rot=0$ that have different LOSS invariants.
\end{theorem}
When $m=1$ this gives back  the twist knot $E_n$ and the corresponding Legendrian representatives are the ones from \cite{OzsvathStipsiczContactsurgeryandtransverseinvariant}. We have a theorem analogous to Theorem \ref{negative rational contact surgery on twist knots} for these representatives of $E(m,n)$. 
\begin{theorem} \label{negative rational contact surgery on two bridge knots}
    Let $L_1$ and $L_2$ be two different Legendrian representatives of $E(m,n)$ from Theorem \ref{Legendrian two bridge knots}. Then for any negative rational number $r\neq -m$, contact $r$-surgery on $L_1$ and $L_2$ gives non contact-isotopic manifolds with different contact invariant.
\end{theorem}
\subsection*{Organization} In \edit{Sections} \ref{sec:hfpreli} and \ref{sec:contactpreli} we review the preliminaries for the Heegaard Floer homology and contact surgeries. After collecting computational results in Section \ref{sec: compute}, we prove Theorem \ref{negative rational contact surgery on twist knots} in Section \ref{sec:proof}. We prove Theorem \ref{negative rational contact surgery on two bridge knots} in Section \ref{sec: Emn}.
\subsection*{Acknowledgement}
The authors would like to thank John Baldwin, Marc
Kegel, and Tom Mark for useful suggestions about the preprint. 
Hugo Zhou thanks Max Planck Institute for Mathematics for its support. Shunyu Wan was supported in part by grants from the NSF (RTG grant DMS-1839968) and the Simons Foundation (grants 523795 and 961391 to Thomas Mark). We also thank the annonymous referee for many helpful comments and suggestions.




\section{Heegaard Floer Homology Preliminaries} \label{sec:hfpreli}
We provide preliminaries for Heegaard Floer homology in this section. The goal is to introduce the dual knot surgery formula in Section 
\ref{subsec: dualknot}.
\subsection{Heegaard Floer invariants for three-manifolds and knots}
Heegaard Floer homology is defined by Ozsv\'{a}th-Szab\'{o} in \cite{OSht}. To a closed oriented $3$-manifold $Y$ with a fixed basepoint $z,$ they associate a \edit{chain complex} $\CF^\circ (Y)$ with four different flavors $\circ = \wedge,+,-$ and $\infty$, \edit{whose chain homotopy equivalence class is a topological invariant of $(Y,z),$} called the \emph{Heegaard Floer chain complex} (we will in general suppress the basepoint from the notation). The generators for $\CF^\circ (Y)$ are given by the intersections of two Lagrangian tori in the symmetric product of the Heegaard surface and the differentials are given by the count of certain holomorphic disks between generators. The \edit{chain complex}  $\CFa(Y)$ is defined over $\F,$ where $\F = \Z/2\Z$ is the field with two elements; \edit{the chain complexes} $\CFm(Y)$ and $\CFp(Y)$ are defined  over the ring $\F[U]$ and $\CFi(Y)$ is a chain complex over the ring $\F[U,U^{-1}];$ each complex admits a Maslov grading.    By taking the homology of $\CF^\circ (Y)$, we obtain the modules $\HF^\circ(Y)$, called the \emph{Heegaard Floer homology of $Y$}.  

Given a knot $K\subset Y$, the Heegaard Floer complexes admit a refinement to a knot invariant, \edit{defined by Ozsv\'{a}th  and Szab\'{o} in \cite{OSknot} and independently by Rasmussen in \cite{Rfk}}.   For each pair $(Y,K), $  by adding a basepoint $w$  which encodes the knot information,
one  imposes an $(i,j)$ \edit{double filtration} over the original chain complex $\CF^\infty (Y)$. The resulting chain complex is graded and doubly-filtered; denoted by $\CFKi(Y,K),$ this is  called the \emph{full knot Floer chain complex}, since other versions of the knot invariants can be obtained as a sub/quotient complex of it. The complex  $\CFKi(Y,K)$ can be viewed as \edit{generated} by formal elements $x=[\x,i,j],$ where $\x$ denotes an intersection point of the Lagrangian tori, and the filtration level $(i,j)$ can be \edit{viewed} as the \emph{\edit{coordinates}} of $x$ over an $(i,j)$-plane.

There are a couple of other versions of the knot complexes that we will use. Define
\begin{align*}
\CFKa(Y,K) &= \{[\x,i,j]\in \CFKi(Y,K) \mid j= 0\}\\
  \CFKm_s(Y,K) &= \{[\x,i,j]\in \CFKi(Y,K) \mid i\leq 0, j=s\} 
\end{align*}
with the differentials inherited from $\CFKi(Y,K)$.  Let $\HFKm_s(Y,K) $ and $\HFKa(Y,K)$ be the homology of $\CFKa(Y,K)$ and $\CFKm_s(Y,K)$ respectively and define $\HFKm (Y,K)= \oplus_{s\in\Z} \HFKm_s(Y,K).$ The above notations follow  the conventions in \cite{OzsvathStipsiczContactsurgeryandtransverseinvariant}. 

The \edit{set} of $\Spin^c$ structures of a closed $3$-manifold $Y$ admits a non-canonical \edit{bijection to $H_1(Y;Z)$}. The Heegaard Floer chain complexes and knot Floer chain complexes split over  \edit{$\Spin^c(Y)$}. Namely, 
\[
\CF^\circ(Y) = \bigoplus_{\spincs \in \Spin^c(Y)} \CF^\circ(Y,\spincs)  \hspace{3em}  \CFK^\circ(Y,K) = \bigoplus_{\spincs \in \Spin^c(Y)} \CFK^\circ(Y,K,\spincs).
\]
\subsubsection{\edit{Strong} invariance and naturality.} \label{subsubsec:stronginv}
In \cite{OSht}, Ozsv\'{a}th  and Szab\'{o} proved that the Heegaard Floer chain complex $\HF^\circ (Y)$ is an invariant of the three-manifold $Y$ up to graded isomorphism. In \cite{JTZnmh}, it was proved that  $\HF^\circ (Y, z)$ is an invariant of the based three-manifold $(Y,z)$ in the following strong sense. In fact, $\HF^\circ (Y,z)$ is a well-defined group, not just an isomorphism class of \edit{groups: diffeomorphisms induce isomorphisms and furthermore  isotopic diffeomorphisms induce identical maps on $\HF^\circ (Y,z)$.}

Similarly,  $\HFKm(Y,K) $ and $\HFKa(Y,K)$ are  invariants of the pair $(Y,K) $ up to graded isomorphism by \cite{OSknot}. It was proved in  \cite{JTZnmh} that in fact they are invariants of the based pair $(Y,K, z,w) $ in the above strong sense. 

The \edit{strong invariance} is crucial in our arguments. Viewing the Heegaard Floer invariants as  actual groups allows us to distinguish certain \edit{elements} (that might have been indistinguishable up to isomorphism). We adopt this perspective \edit{throughout} the paper, specifically in the proofs in Section \ref{sec:proof} and \ref{sec: Emn}.

\subsection{Rational surgery and mapping cone formula}
For a null-homologous knot $K$ \edit{in a rational homology sphere $Y$}, 
Ozsv\'{a}th  and Szab\'{o} \edit{described}  in \cite{rational} an algorithm that computes $\CFi(Y_{r}(K))$ with $r\in \Q$  using the input of
the knot Floer complex $\CFKi(Y,K)$ \edit{together with the ``flip map"}. We will now describe this algorithm.

Let $p, q$ be a pair of coprime integers such that $ q> 0.$
Given  $C=\CFKi(Y,K)$,  a  finitely generated, graded chain complex over the ring $\F[U,U^{-1}]$ with $(i,j)$ \edit{double filtration}. For $t\in \Z$, let $(t,A_s(C))$ and $(t,B_s(C))$ both denote a copy of $C$ and set $s = \floor{t/q}$.   Define  $v^{\infty}_t : (t, A_s(C)) \rightarrow (t,B_s(C))$ to be the identity map and $h^{\infty}_t : (t, A_s(C)) \rightarrow (t+p,B_{s'}(C))$ with $s' = \floor{(t+p)/q}$ to  be the composition $U^s \circ \textit{flip},$ where $\textit{flip} : C \rightarrow C$ denotes the ``flip map". \edit{The flip map is a skew-filtered (it is filtered with respect to the $j$ filtration in the domain and the $i$ filtration in the range) chain homotopy equivalence induced by a sequence of Heegaard moves that exchange the two basepoints. By \cite[Lemma 2.18]{HeddenLevine} such map is unique up to skew-filtered chain homotopy equivalences (in the above sense) when $Y$ is an L-space. Recall that $\CFKi(Y,K)$ is symmetric with respect to exchanging coordinates $i$ and $j$ \cite[Proposition 3.10]{OSknot}.       Therefore when $Y=S^3$, we can simply pick the flip map  to be the reflection  along the line $i=j$.} For $i\in \Z/p\Z,$ define $(\X^{\infty}_{p/q}(C),i)$ to be the mapping cone of
 \begin{align} \label{eq:x_infinity1_rational}
      \bigoplus^{gq-1}_{\substack{t=-(g-1)q \\ t \equiv i \mod |p|}}(t,A_s(C)) \xrightarrow{v^{\infty}_t+h^{\infty}_t} \bigoplus^{gq-1}_{\substack{t=-(g-1)q -1 + p\\ t \equiv i \mod |p|}}(t,B_s(C)),
\end{align}
where $g=g(K).$ Denote also $\X^{\infty}_{p/q}(C) = \bigoplus_{i \in \Z/p\Z}(\X^{\infty}(C),i)$. 
When  the surgery coefficient is clear,  we suppress $p/q$ from the notations and simply write $\X^{\infty}(C)$. When $C=\CFKi(Y,K), $ we can also write $\X^{\infty}(Y,K)$ to specify the manifold-knot pair. 
\begin{theorem}[\cite{rational}]
     \edit{There is an identification $\Spin^{c}(Y_{p/q}(K)) \cong \Z/p\Z,$ such that for each $i\in \Spin^{c}(Y_{p/q}(K))$,} the complex $\CFi(Y_{p/q}(K),i)$ is chain homotopy equivalent to $(\X^{\infty}_{p/q}(Y,K),i)$.
\end{theorem}
When $q=1,$ namely for the integer framed surgery on $K,$ we can simply write $A_s$ and $B_s$ in the place of $(t,A_s(C))$ and $(t,B_s(C))$. In the case of $n$-surgery for some $n\in \Z,$ $\X^{\infty}(C)$ is the mapping cone of 
\begin{align} \label{eq:x_infinity1}
      \bigoplus^{g-1}_{s=-g+1}A_s(C) \xrightarrow{v^{\infty}_s+h^{\infty}_s} \bigoplus^{g-1}_{s=-g+n+1}B_s(C),
\end{align} 
which similarly splits over $\Spin^c$ structures. The same algorithm works for the hat version as well. Recall that there is an $\II$-filtration over $\X^{\infty}(C)$, where for $[\x,i,j] \in (t, A_{s}(C))$,
\[
\II([\x,i,j]) = \max\{i,j-s\} 
\]
and for $[\x,i,j] \in (t,B_{s}(C))$,
\[\II([\x,i,j]) = i.\]
If we define $\Xa(C) = \X^\infty (C)|_{\II=0},$ then the complex $\CFa(Y_{p/q}(K))$ is chain homotopy equivalent to $\Xa(Y,K)$. We denote $\hat{A}_s = A_s |_{\II=0}, \hat{B}_s = 
 B_s |_{\II=0}$ and denote the induced maps by $\hat{v}_s$ and $\hat{h}_s$ respectively.
\subsection{Dual knot surgery formula} \label{subsec: dualknot}
In \cite{HeddenLevine}, Hedden and Levine defined a refinement of the above construction, which for $n \neq 0$ outputs  $\CFKi(S_n^3(K),\mu)$, the knot Floer chain complex of the meridian (dual knot) of $K$ in the surgery. In fact this construction extends to rational surgeries on knots inside rational homology spheres. For the sake of the simplicity, we record here only the case of integer surgery on knots in $S^3.$ 
Following the previous conventions, define $\X^{\infty}_K(C)$ to be the mapping cone of
 \begin{align} \label{eq:x_infinity2}
      \bigoplus^{g}_{s=-g+1}A_s(C) \xrightarrow{v^{\infty}_s+h^{\infty}_s} \bigoplus^{g}_{s=-g+n+1}B_s(C).
\end{align}
and define the double-filtration $(\II, \JJ)$ and the Maslov grading over $\X^{\infty}_K(C)$ as follows.
\begin{align}
\intertext{For $[\x,i,j] \in A_{s}(C)$,}
\label{eq: filtration_s3_1}
 \II([\x,i,j]) &= \max\{i,j-s\} \\
 \label{eq: filtration_s3_2}
 \JJ([\x,i,j]) &= \max\{i-1,j-s\} + \frac{2s+n-1}{2n} \\
\label{eq: grt-def-A} \gr([\x,i,j]) &= \widetilde{\gr}([\x,i,j]) + \frac{(2s-n)^2}{4n} + \frac{2-3\operatorname{sign}(n)}{4}  
\intertext{and for $[\x,i,j] \in B_{s}(C)$,}
 \label{eq: filtration_s3_3}
 \II([\x,i,j]) &= i \\
  \label{eq: filtration_s3_4}
 \JJ([\x,i,j]) &= i-1 + \frac{2s+n-1}{2n}
  \\ \label{eq: grt-def-B} \gr([\x,i,j]) &= \widetilde{\gr}([\x,i,j]) + \frac{(2s-n)^2}{4n} + \frac{-2-3\operatorname{sign}(n)}{4}
\end{align}
where $ \widetilde{\gr}$ indicates the Maslov grading  in the original complex.
Collapsing the $\JJ$-filtration in $\X_K^{\infty}(C)$ recovers the  complex $\X^{\infty}(C)$ in the original construction by Ozsv\'{a}th  and Szab\'{o}. In other words, $\X_K^{\infty}(C)$ is chain homotopy equivalent to $\X^{\infty}(C)$ as \edit{an} unfiltered chain complex.
\begin{theorem}[\cite{HeddenLevine}]
    The complex $\CFKi(S^3_{n}(K),\mu)$ is filtered chain homotopy equivalent to $\X_K^{\infty}(\CFKi(S^3,K))$, where $\mu$ is the image of the meridian of $K$ in the surgery.
\end{theorem}


\section{Contact surgery}\label{sec:contactpreli}
In this section we talk about the preliminaries on contact surgery, contact invariant and LOSS invariant. We will state the naturality theorems of those invariants in Section \ref{subsec: Naturality}, which will play a crucial role in the proofs. 
\subsection{Contact surgery and DGS algorithm} \label{general contact surgery}
Given an oriented Legendrian knot $L$ in a contact \edit{3-manifold} $(Y,\xi)$ there exist a contact framing defined by a vector field along $L$ that is always transverse to the 2-plane field $\xi$. In \cite{DGst} Ding and Geiges define a notion of contact $r$-surgery on $L$ which, for a choice of rational number \edit{$r\neq0$}, gives rise to another contact \edit{3-manifold} $(Y',\xi_r(L))$. (In general there are choices \edit{involved} to completely determine the resulting contact structure for a general contact rational $r$-surgery. \edit{In this paper, whenever such a choice is required, we will always take the negative stabilization.}) In \cite{DGS} Ding, Geiges and Stipsicz prove the following theorems. 

\begin{theorem}
    Every (closed, orientable) contact \edit{3-manifold} $(Y,\xi)$ can
be obtained via contact $(\pm 1)$-surgery on a Legendrian link in $(S^3,\xi_{std})$.
\end{theorem}

Moreover they describe an algorithm to transform any rational $r$-surgery on a Legendrian knot $L$ \edit{into} a sequence of $(\pm 1)$-surgeries on some (stabilizations of) push-offs of $L$. The procedure can be divided into the following two theorems that correspond to the two cases when $r<0$, and $r>0$ respectively.

\begin{theorem}[DGS algorithm for $r<0$ \cite
{DGS}] {\label{DGS negative ration contact surgery}}
Given a Legendrian knot $L$ in $(Y,\xi)$. Let $0>-x/y=r\in \mathbb{Q}$ be a contact surgery coefficient with continued fraction \begin{equation}
    -\frac{x}{y}=[a_1+1,a_2,\edit{\dots},a_\l]^{-}=a_1+1-\cfrac{1}{a_2-\cfrac{1}{\edit{\dots}-\cfrac{1}{a_\l}}}
\end{equation}
where each $a_i\leq -2$. Then any contact $(x/y)$-surgery on $L$ can be described as contact surgery along a link $L_1 \cup\edit{\dots}\cup L_{\ell}$, where 
\begin{itemize}
     
    \item $L_1$ is obtained from $L$ by stabilizing $|a_1+2|$ times.
    \item $L_j$ is obtained from a Legendrian push-off of $L_{j-1}$ by stabilizing $|a_j+2|$ times, for $j\geq 2$.
    \item The contact surgery coefficient \edit{is} $-1$ on each $L_j$.
\end{itemize}
    
\end{theorem}

\begin{theorem}[\cite
{DGS} DGS algorithm for $r>0$] {\label{DGS positive contact srugery}}
Given a Legendrian knot $L$ in $(Y,\xi)$. Let $0<x/y=r\in \mathbb{Q}$ be a contact surgery coefficient. Let $e\in \mathbb{Z}$ be the minimal positive integer such that $\frac{x}{y-ex}<0$, with the continued fraction \begin{equation}
    \frac{x}{y-ex}=[a_1,a_2,\edit{\dots},a_\l]=a_1-\cfrac{1}{a_2-\cfrac{1}{\edit{\dots}-\cfrac{1}{a_\l}}}
\end{equation}
where each $a_j\leq -2$. Then any contact $(x/y)$-surgery on $L$ can be described as contact surgery along a link $(L_0^1 \cup L_0^2 \cup \edit{\dots} \cup L_0^e)\cup L_1 \cup\edit{\dots}\cup L_{\ell}$, where 
\begin{itemize}
     \item $L_0^1,\edit{\dots},L_0^e$ are Legendrian push-offs of $L$.
    \item $L_1$ is obtained from a Legendrian push-off of $L_0^e$ by stabilizing $|a_1+1|$ times.
    \item $L_j$ is obtained from a Legendrian push-off of $L_{j-1}$ by stabilizing $|a_j+2|$ times, for $j\geq 2$.
    \item The contact surgery coefficients are $+1$ on each $L_0^i$ and $-1$ on each $L_j$.
\end{itemize}
    
\end{theorem}

The choices we mentioned in the beginning of the section correspond to the choices of stabilization for each $L_j$, each of which can be either positive or negative. As we mentioned in the introduction we only consider the case when the stabilizations are all negative. The most important cases we are interested in are the negative $(-n)$-surgery and positive $(\frac{n+1}{n})$-surgery (we will consider general $-x/y$ later). If we follow the above two algorithms carefully and only consider negative stabilization, the contact \edit{$(-n)$-surgery} on a Legendrian knot $L$ is the same as doing contact $(-1)$-surgery along the $n-1$ times negatively stabilized $L$, and the contact \edit{$(\frac{n+1}{n})$-surgery} on $L$ is the same as doing contact surgery along link $L \cup L_1$ where $L_1$ is the $n$ times negative stabilization of a Legendrian push-off of $L$ with \edit{contact surgery coefficient} $+1$ on $L$, and $-1$ on $L_1$. 

\subsection{Contact invariants and LOSS invariants }
Given a contact 3-manifold $(Y,\xi)$ Ozsv\'{a}th-Szab\'{o} \cite{OSHFandcontactstructures} and later Honda-Kazez-Mati\'c \cite{HKMcontactclassinHF} showed that $(Y,\xi)$ determines a distinguished element $c(\xi)\in \HFa(-Y)$, called the Heegaard Floer contact invariant. Subsequently, for a Legendrian knot $L$ in $(Y,\xi)$, Lisca-Ozsv\'ath-Stipsicz-Szab\'o defined 
the ``LOSS invariant'' $\mathfrak{L}(L) \in \HFKm(-Y,L)$ \cite{LOSS}. We refer the reader to the above papers for precise definitions. 

For any 3-manifold $Y$ and a knot $K$ in $Y$ there is a natural chain map 
$$g: \CFKm(Y,K,\mathfrak{t}) \rightarrow \CFa(Y,\mathfrak{t})$$
\edit{obtained by} setting $U=1$. 
The map $$ G: \HFKm(Y,K,\mathfrak{t}) \rightarrow \HFa(Y,\mathfrak{t})$$ is the homology map induced by $g$. The LOSS invariant is related to the contact invariant in the following way.

\begin{lemma}[\cite{LOSS}] \label{G maps LOSS to Contact invariant}
    Let $L$ be an oriented null-homologous Legendrian knot in a contact 3-manifold $(Y,\xi)$, then the map \begin{equation} \label{LOSS to contact invariant}
        G: \HFKm(-Y,L,\mathfrak{t}) \rightarrow \HFa(-Y,\mathfrak{t})
    \end{equation} has the property that $$G(\mathfrak{L}(L))=c(\xi).$$
\end{lemma}

Another important property of the LOSS invariant is   that it is  unchanged under negative stabilization.

\begin{theorem}[\cite{LOSS}] \label{LOSS invariant under negative stabilization}
Suppose that $L$ is an oriented Legendrian knot and denote the negative  stabilizations of $L$ as $L^-$  Then,
$\mathfrak{L}(L^-) = \mathfrak{L}(L)$.
\end{theorem}

\subsection{Naturality of contact invariant and LOSS invariant under contact surgery} \label{subsec: Naturality}
\edit{Recall that given a rational number $p/q$ with $p,q$ relatively prime and $0<q$, and write $p=mq-r$ for some integer $m$ and rational number $r$ with $0\leq r<q$. Then $p/q$ surgery on $K$ in a rational homology sphere $Y$ induces a \emph{rational surgery cobordism} $W:Y\#{-L(q,r)} \rightarrow Y_{p/q}(K)$, where $Y\#{-L(q,r)}$ is the result of $(q/r)$-surgery on a meridian of $K$, and $W$ is given by attaching a $2$-handle along
the image of $K$ after this surgery with framing $m$.}
\begin{theorem}[{\cite[Theorem 1.1]{MarkBulentNaturalityofContactinvariant}}] \label{naturality of contact invariant}
    Let $L$ be an oriented null-homologous Legendrian knot in a contact rational homology sphere $(Y,\xi)$ with non-vanishing contact invariant $c(\xi)$, let $0<x/y\in \mathbb{Q}$ be the contact surgery coefficient and $p/q=\tb(L)+x/y$ be the corresponding smooth surgery coefficient. Let $W:Y\#{-L(q,r)} \rightarrow Y_{p/q}(L)$ be the corresponding rational surgery cobordism, where $p=mq-r$ with $0\leq r<q$, and consider $\xi_{x/y}(L)$ on $Y_{p/q}(L)$ (When we write smooth surgery on a Legendrian knot $L$ we meant to view $L$ as it smooth knot type).
    \begin{enumerate}
        \item There exists a $\Spin^c$ structure $\mathfrak{s}$ on $W$ and a generator $\tilde{c}\in \HFa(L(q,r),\edit{\mathfrak{s}|_{ L(q,r)}})$ such that the homomorphism 
        
        $$F_{-W,\mathfrak{s}}: \edit{\HFa(-Y \# L(q,r),\mathfrak{s}|_{-Y \# L(q,r)}) \rightarrow \HFa(-Y_{p/q}(L),\mathfrak{s}|_{-Y_{p/q}(L)})}$$ 
        
        induced by $W$ with its orientation reversed satisfies 
        
        $$F_{-W,\mathfrak{s}}(c(\xi)\otimes \tilde{c})= c(\xi_{x/y}(L)).$$
        
        \item Moreover if both $\xi$ and $\xi_{x/y}(L)$ have torsion first Chern class\edit{, then} $\mathfrak{s}$ has the property that
        
        $$\pm \langle c_1(\mathfrak{s}),[\Tilde{Z}] \rangle =p+(\rot(L)-\tb(L))q-1$$
 where $Z$ is a Seifert surface for $L$ and $[\Tilde{Z}]$ is the result of capping off $Z$ with $q$ parallel \edit{copies} of the core of the handle in $W$.
    \end{enumerate}
\end{theorem}

We will analyse the map $F_{-W,\mathfrak{s}}$ in the above theorem to see when this map is injective. More specifically, we will study the corresponding map using the mapping cone formula given by the following proposition. 

\begin{proposition}[{\cite[Corollary 1.5]{MarkBulentNaturalityofContactinvariant}}] \label{mapping cone for contact invariant}
    Let $L$ be an oriented Legendrian knot in a contact integer homology sphere $(Y,\xi)$, fix $0< x/y\in \mathbb{Q}$ to be a contact surgery coefficient corresponding to smooth surgery coefficient $p/q=\tb(L)+x/y$. 
    Let $W:Y\#{-L(q,r)} \rightarrow Y_{p/q}(L)$ be the corresponding rational surgery cobordism, where $p=mq-r$.
    Then the contact invariant $c(\xi_{x/y}(L))\in \HFa(-Y_{p/q}(L))$ is equal (up to conjugation of the $\Spin^c$ structure on the cobordism) to the image of $c(\xi)$ in homology of the map given by inclusion 
    $$(t,\hat{B}_s)\hookrightarrow \Xa_{-p/q}(-Y,L),$$
    where we \edit{view $c(\xi)$ as an element in $(t,\hat{B}_s)$ under the identification between $(t,\hat{B}_s)$ and $\CFa(-Y)$}, and $t$ satisfies $$2t=(\rot(L)-\tb(L)+1)q-2.$$
\end{proposition}

\begin{remark}
    Mark-Tosun only state the above proposition for $(S^3,\edit{\xi_{std}})$, but their proofs extend naturally \edit{to} the situation when $Y$ is an integer homology sphere, and even for null-homologous \edit{knots} in rational homology sphere.
\end{remark}

Note that in both Theorem \ref{naturality of contact invariant} and Proposition \ref{mapping cone for contact invariant} the $\Spin^c$ structures are only determined up to conjugation. However, such sign (conjugation) ambiguity can be removed, given that we are doing positive integer contact surgery on (rationally) null-homologous Legendrian knot, \edit{and that} both $Y$ and the manifold after surgery $Y'$ are rational homology spheres. \edit{This is due to the following technical result, which} will end up playing an important role in the proof of Theorems \ref{negative rational contact surgery on twist knots} and \ref{negative rational contact surgery on two bridge knots}.

\begin{theorem}[{\cite[Theorem 6.3]{WanNaturalityofLOSSinvariant}}] \label{thm: same Spin^c for Legendrian with classical invariants}
    Let $L$ be an oriented rationally null-homologous Legendrian knot in a contact rational homology sphere $(Y,\xi)$ with non-vanishing contact invariant $c(\xi)$. Let $0<n \in \mathbb{Z}$ be the contact surgery coefficient, $Y'$ be the manifold after contact \edit{$(+n)$-surgery} on $L$, and let $W:Y \rightarrow Y'$ be the corresponding  surgery cobordism, and consider $\xi_{n}^-(L)$ on $Y'$. There exist a $\Spin^c$ structure $\mathfrak{s}$ on $W$ such that the homomorphism 
        
        $$F_{-W,\mathfrak{s}}: \HFa(-Y) \rightarrow \HFa(-Y')$$ 
        induced by $W$ with its orientation reversed satisfies 
        $$F_{-W,\mathfrak{s}}(c(\xi))= c(\xi_n^-(L)).$$
Moreover, if $Y'$ is also a rational homology sphere. Then the $\Spin^c $ structure $\mathfrak{s}$ has the property that $$ \langle c_1(\mathfrak{s}),[\Tilde{F}] \rangle = 
      y(\edit{\rot}_\mathbb{Q}(L)+n-1) $$
 where $y$ is the order of $[L]$, $F$ is a rational Seifert surface for $L$ and $\Tilde{F}$ is the ``capped off" surface of $F$.
\end{theorem}

We also have the parallel naturality theorem for the LOSS invariant which  will also be used later.

\begin{theorem}[{\cite[Theorem 1.1]{OzsvathStipsiczContactsurgeryandtransverseinvariant}}]{\label{Naturality of LOSS}}
Let $L,S \in (Y, \xi)$ be two disjoint oriented Legendrian knots in the contact 3-manifold $(Y, \xi)$ with $L$ null-homologous. Let $(Y',\xi_1)$ denote the 3-manifold with the associated contact structure obtained by performing contact $(+1)$-surgery along $S$, and let $L'$ denote the oriented Legendrian knot which is the image of $L$ in $(Y',\xi_1)$. Moreover suppose that $L'$ is null-homologous in $Y'$. Let $W$ be the 2-handle cobordism from $Y$ to $Y'$ induced by the surgery, and let
\begin{equation*}
    F_{-W,\mathfrak{s}}: \HFKm (-Y,L) \rightarrow \HFKm(-Y',L')
\end{equation*}
be the homomorphism in knot Floer homology induced by $-W$, the cobordism with reversed orientation, for  $\mathfrak{s}$ a $\Spin^c$ structure on $-W$. Then

\begin{enumerate}
    \item \label{it: naturality1}  there is a unique choice of \s \ for which 
\begin{equation*}
    F_{-W,\mathfrak{s}}(\mathfrak{L}(L))=\mathfrak{L}(L')
\end{equation*}
holds, and for any other $\Spin^c$ structure \s \  the map $F_{-W,\mathfrak{s}}$ is trivial on $\mathfrak{L}(L)$. 

\item\label{it: naturality2}  \cite[Proposition 1.4]{WanNaturalityofLOSSinvariant} If $S$ is null-homologous and both $Y$ and $Y'$ are rational homology \edit{spheres} then $\mathfrak{s}$ has the property that $$ \langle c_1(\mathfrak{s}),[\Tilde{Z}] \rangle =\rot(S)$$
where $Z$ is a Seifert surface for $S$ and $\Tilde{Z}$ is the result of capping off $Z$ with the core of the handle in $W$.
\end{enumerate}

\end{theorem}

\section{A filtered mapping cone computation}\label{sec: compute}
In this section we perform the computations used in the proofs of the main theorems.

The knot $E_n$ has the two-bridge notation $[-2,n+1]$, which can be viewed as the numerical closure of the rational tangle $-2 + \frac{1}{n+1} = -\frac{2n+1}{n+1}.$ We will be interested in the (knot Floer \edit{homology of the}) mirror $-E_n,$  which has the rational tangle number $\frac{2n+1}{n+1}.$ For two-bridge knots, \edit{the signature,}  Alexander polynomials and $\HFKa$ \edit{can be computed explicitly from} the rational tangle number (see for example \cite{uber, twobridge}). \edit{In particular $\sigma(E_n) = -2.$} Together with the classification theorem for thin knots \cite{OSalter, inathin}, we compute
\begin{align*}
    \CFKi(-E_n) \cong C \oplus \big( \bigoplus_{i=1}^\frac{n-1}{2} D_i \big)
\end{align*}
where $C$ is isomorphic to $\CFKi(-T_{2,3})$ and each $D_i$ is a length-one box summand (compare with the homology $\HFKm(-E_n)$ and $\HFKa(-E_n)$ calculated in \cite{OzsvathStipsiczContactsurgeryandtransverseinvariant}). We pick the basis as follows. Let $x, y, z$ be the generators of $C$ over $\F[U],$ such that in the $(i,j)$-coordinate of $\CFKi,$ they have \edit{coordinates} $(0,1), (0,0)$ and $(1,0)$, with \edit{Maslov} grading $2,1$ and $2$ respectively. The differentials are given by 
\begin{align*}
    \d x = \d y =z.
\end{align*}
Each $D_i$ is isomorphic to $D,$ generated by $a, b, c$ and $d$ over $\F[U],$ with \edit{coordinates} $(1,1),(0,1),(1,0)$ and $(0,0)$ and \edit{Maslov} grading $3,2,2$ and $1$ respectively. \edit{The} differentials are given by 
\begin{align*}
    \d a &= b + c \\
    \d b &= \d c =d.
\end{align*}

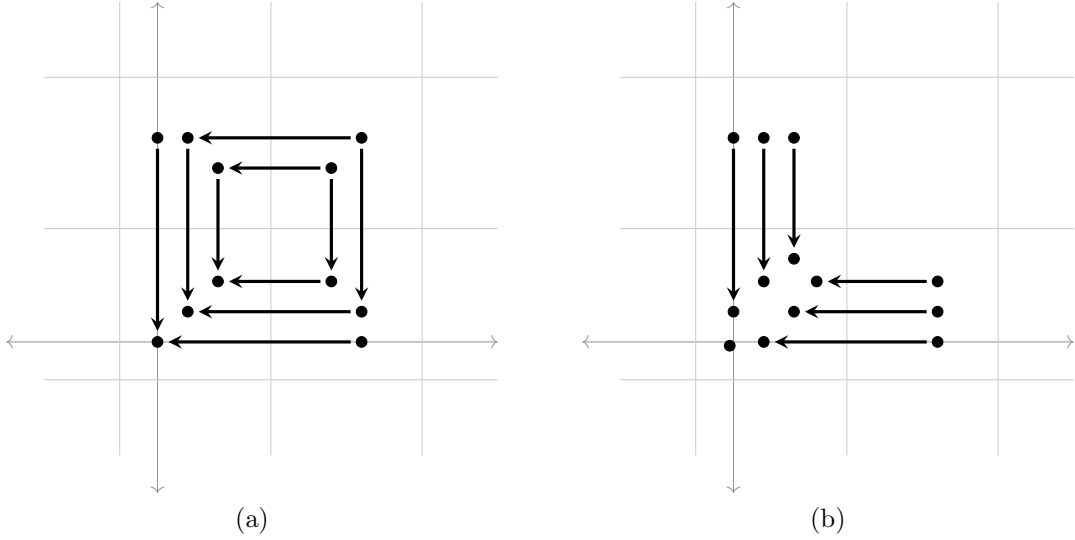
\begin{figure}[htb!]
\centering
  \begin{minipage}{0.5\linewidth}
\centering
\subfloat[]{
\begin{tikzpicture}
\begin{scope}[thin, black!40!white]
		\draw [<->] (-1.5, 0.5) -- (5, 0.5);
		\draw [<->] (0.5, -1.5) -- (0.5, 5);
	\end{scope}
\begin{scope}[thin, black!20!white]
    \foreach \x in {0, 2, 4}
         { \draw   (\x, -1) -- (\x, 5);
         \draw  (-1, \x) -- (5,\x);}
\end{scope}
    \filldraw (2.8, 2.8) circle (2pt) node[] (a){};
    \filldraw (2.8, 1.3) circle (2pt) node[] (b){};
    \filldraw (1.3, 2.8) circle (2pt) node[] (c){};
    \filldraw (1.3, 1.3) circle (2pt) node[] (d){};
    \draw [very thick, -stealth] (a) -- (b);
    \draw [very thick, -stealth] (a) -- (c);
    \draw [very thick, -stealth] (c) -- (d);
    \draw [very thick, -stealth] (b) -- (d);

     \filldraw (3.2, 3.2) circle (2pt) node[] (a1){};
    \filldraw (3.2, 0.9) circle (2pt) node[] (b1){};
    \filldraw (0.9, 3.2) circle (2pt) node[] (c1){};
    \filldraw (0.9, 0.9) circle (2pt) node[] (d1){};
    \draw [very thick, -stealth] (a1) -- (b1);
    \draw [very thick, -stealth] (a1) -- (c1);
    \draw [very thick, -stealth] (c1) -- (d1);
    \draw [very thick, -stealth] (b1) -- (d1);

    \filldraw (3.2, 0.5) circle (2pt) node[] (b11){};
    \filldraw (0.5, 3.2) circle (2pt) node[] (c11){};
    \filldraw (0.5, 0.5) circle (2pt) node[] (d11){};

    \draw [very thick, -stealth] (c11) -- (d11);
    \draw [very thick, -stealth] (b11) -- (d11);
       \end{tikzpicture}
}
\end{minipage}%
  \begin{minipage}{0.5\linewidth}
\centering
\subfloat[]{
\begin{tikzpicture}
\begin{scope}[thin, black!40!white]
		\draw [<->] (-1.5, 0.5) -- (5, 0.5);
		\draw [<->] (0.5, -1.5) -- (0.5, 5);
	\end{scope}
\begin{scope}[thin, black!20!white]
    \foreach \x in {0, 2, 4}
         { \draw   (\x, -1) -- (\x, 5);
         \draw  (-1, \x) -- (5,\x);}
\end{scope}
  
    \filldraw (3.2, 1.3) circle (2pt) node[] (b){};
    \filldraw (1.3, 3.2) circle (2pt) node[] (c){};
    \filldraw (1.3, 1.6) circle (2pt) node[] (d){};
    \filldraw (1.6, 1.3) circle (2pt) node[] (e){};

    \draw [very thick, -stealth] (c) -- (d);
    \draw [very thick, -stealth] (b) -- (e);
     \filldraw (0.45, 0.45) circle (2pt) node[] (a1){};
    \filldraw (3.2, 0.9) circle (2pt) node[] (b1){};
    \filldraw (0.9, 3.2) circle (2pt) node[] (c1){};
    \filldraw (0.9, 1.3) circle (2pt) node[] (d1){};
 \filldraw (1.3, 0.9) circle (2pt) node[] (e1){};
 
    \draw [very thick, -stealth] (c1) -- (d1);
    \draw [very thick, -stealth] (b1) -- (e1);

    \filldraw (3.2, 0.5) circle (2pt) node[] (b11){};
    \filldraw (0.5, 3.2) circle (2pt) node[] (c11){};
    \filldraw (0.5, 0.9) circle (2pt) node[] (d11){};
       \filldraw (0.9, 0.5) circle (2pt) node[] (e11){};

    \draw [very thick, -stealth] (c11) -- (d11);
    \draw [very thick, -stealth] (b11) -- (e11);
       \end{tikzpicture}
}
\end{minipage}
    \caption{On the left, the knot Floer complex of $-E_5=7_2.$ On the right, the knot Floer complex of its dual knot, $\CFKi(S^3_{+1}(-E_5),\mu).$ }
    \label{fig:knotfloercomplex}
\end{figure} 

\begin{proposition} \label{prop:dualEn} The dual knot complex of $-E_n$ is given by
    \begin{align*} 
        \CFKi(S^3_{+1}(-E_n),\mu) \cong O \oplus \Big(\bigoplus^{\frac{n+1}{2}}_{i=1} H_i \Big) \oplus \Big(\bigoplus^{\frac{n+1}{2}}_{i=1} V_i \Big)  
    \end{align*} 
    where $O$ is the complex with a single generator, supported in \edit{coordinates} $(0,0)$. Each  $H_i$ is generated by $x^h_i, y^h_i$ in \edit{coordinates} $(0,0)$ and $(1,0)$ respectively, with  differential  $\d y^h_i = x^h_i.$ Each  $H_i$ is generated by $x^v_i, y^v_i$ in \edit{coordinates} $(0,0)$ and $(0,1)$ respectively, with differential  $\d y^v_i = x^v_i.$ 
\end{proposition}
\begin{proof}
We first compute 
 $\X^{\infty}_K(C)$ with surgery coefficient \edit{$\lambda=1$}, which by \eqref{eq:x_infinity2} is the mapping cone of
 \begin{align} 
      A_0(C) \oplus A_1(C) \xrightarrow{v_1+h_0} B_1(C).
\end{align}
Recall that $A_0(C), A_1(C)$ and $B_1(C)$ are each isomorphic to a copy of $C.$   We denote the generators in $A_s(C)$ by $x^{(s)}, y^{(s)}, z^{(s)}$ for $s=0,1$ and the generators in $B_1(C)$ by  $x', y', z'.$ \edit{We temporarily denote the inner differentials of $A_s$ and $B_s$ by $\d^0$, in order to distinguish them from the mapping cone differentials, which we denote by $\d$. We have  $  \d = \d^0 +  h_0 $ on $A_0$, $  \d = \d^0 +  v_1 $ on $A_1$ and $  \d = \d^0 $ on $B_1$,   where 
}
\begin{align*}
    v_1(x^{(1)}) = x'   \qquad v_1(y^{(1)}) &= y' \qquad v_1(z^{(1)}) = z' \\
     h_0(x^{(0)}) = y'   \qquad h_0(y^{(0)}) &= x' \qquad h_0(z^{(0)}) = z'.
\end{align*}
Note that in the $(\II,\JJ)$-filtration defined by \eqref{eq: filtration_s3_1}, \eqref{eq: filtration_s3_2}, \eqref{eq: filtration_s3_3} and \eqref{eq: filtration_s3_4}, $x'$ and $z'$ are in the same $(\II,\JJ)$-coordinate. Therefore quotienting out \edit{the $\F[U,U^{-1}]$-submodule spanned by $\{ x', \d x' \}$} yields a chain homotopy equivalent complex. Similarly we can quotient out \edit{the $\F[U,U^{-1}]$-submodule spanned by $\{ x^{(0)}, \d x^{(0)}  \}$}. The resulting complex is generated by $\{ y^{(0)}, \edit{z^{(0)} (=y'),  x^{(1)}, y^{(1)}}, z^{(1)} \}$ over $\F[U,\edit{U^{-1}}].$ We can further simplify the complex by a filtered change of basis to $\{ y^{(0)}, z^{(0)},\edit{y^{(0)} + x^{(1)} + y^{(1)}}, x^{(1)}, z^{(1)} \}$, with $(\II,\JJ)$-filtration $(1,0), (0,0), (1,1), (0,1),(0,0)$ respectively, \edit{where the nontrivial differentials are given by}
\begin{align*}
    \d y^{(0)} &= z^{(0)} \\
    \d x^{(1)} &= z^{(1)}.
\end{align*}
\edit{Therefore $\{ y^{(0)}, z^{(0)}\}, \{x^{(1)}, z^{(1)} \}$ and $\{ y^{(0)} + x^{(1)} + y^{(1)}\}$ each generates a summand, say    $H_{\frac{n+1}{2}}$, $V_{\frac{n+1}{2}}$   and $O$, respectively. Note that we can also choose the generator of the complex $O$ to be $U(y^{(0)} + x^{(1)} + y^{(1)})$, which has $(\II,\JJ)$-coordinate $(0,0)$.}

Next we compute 
 $\X^{\infty}_K(D)$ for $D \cong D_i, 1\leq i \leq \frac{n-1}{2}.$ By \eqref{eq:x_infinity2} this is the mapping cone of
 \begin{align} 
      A_0(D) \oplus A_1(D) \xrightarrow{v_1+h_0} B_1(D).
\end{align}
 We adopt \edit{a similar notation} and denote the generators in $A_s(D)$ by $a^{(s)}, b^{(s)}, c^{(s)}, d^{(s)}$ for $s=0,1$ and the generators in $B_1(D)$ by  $a', b', c', d'.$ Note that $a', c'$ and $b',d'$ are pairwise in the same $(\II,\JJ)$-coordinate, and therefore the mapping cone is homotopy equivalent to $ A_0(D) \oplus A_1(D)$. We can further quotient out $\{a^{(0)} \d a^{(0)} \}$ and $\{a^{(1)} \d a^{(1)} \}$. The resulting complex is generated by $\{c^{(0)},d^{(0)},b^{(1)},d^{(1)} \}$ with $(\II,\JJ)$-filtration $(1,0),(0,0),(0,1),(0,0)$ respectively and differentials 
 \begin{align*}
     \d c^{(0)} &= d^{(0)} \\
     \d b^{(1)} &= d^{(1)}.
 \end{align*}
 We have 
    \begin{align*}
        \CFKi(S^3_{+1}(-E_n),\mu) &\cong \X^{\infty}_K(\CFKi(S^3,-E_n))\\
        &\cong \X^{\infty}_K\left(C \oplus \Bigg(\bigoplus_{i=1}^\frac{n-1}{2} D_i \Bigg)\right)\\
        &= \X^{\infty}_K(C)\oplus \left(\bigoplus_{i=1}^\frac{n-1}{2} \X^{\infty}_K(D_i)\right).
    \end{align*}
The result follows from the previous computation. In particular, each $D_i$ contributes \edit{the summands} $V_i$ and $H_i$, and $C$ contributes \edit{the summands $O$, $V_{\frac{n+1}{2}}$ and $H_{\frac{n+1}{2}}$.}
\end{proof}
\begin{corollary}\label{cor: topinj} The map
    \begin{equation*}
     G: \HFKm ({S^3_{+1}}(-E_n),\mu) \rightarrow \HFa ({S^3_{+1}}(-E_n)), 
\end{equation*}
is injective in the top Alexander grading. 
\end{corollary}
\begin{proof}
    Using the notation from Proposition \ref{prop:dualEn}, the only generators supported in the top Alexander grading of $\HFKm ({S^3_{+1}}(-E_n),\mu)$ are \edit{$y^v_1, \dots, y^v_{\frac{n+1}{2}}$}. When $s<0,$  each
    $U^{1-s}y^v_i \in C\{i\leq 0, j=s\} \cong \CFa({S^3_{+1}}(-E_n))$ survives \edit{in} the homology. The result follows.
\end{proof}
\section{Proof of Theorem \ref{negative rational contact surgery on twist knots}} \label{sec:proof}
We will prove the theorem by considering different cases depending on the contact surgery \edit{coefficient} $r$. Recall that $\xi_r(L)$ denotes the contact structure obtained by contact $r$-surgery on the Legendrian knot $L$. 

\subsection{Case for $\pmb{r=-2}$} 

This is the most essential case and the starting point. 
\begin{theorem} \label{-2 contact surgery}
    Let $L_1$ and $L_2$ be two Legendrian representatives of $E_n$ with $n>3$ and odd that have same $\tb=1$ and $\rot=0$ but different LOSS invariants. Then contact \edit{$(-2)$-surgery} on $L_1$ and $L_2$ gives contact manifolds that have different contact invariants.
\end{theorem}

The point of doing the contact $(-2)$-surgery is that smoothly we are doing $(-1)$-surgery, which makes the resulting manifold an integer homology sphere. Thus all knots will \edit{be} null-homologous. The \edit{idea of the proof of} Theorem \ref{-2 contact surgery} is \edit{to first show that} the LOSS invariants of the induced Legendrian push-offs of $L_1$ and $L_2$ are distinct and then use the fact that the map \eqref{LOSS to contact invariant} is injective on the top Alexander grading (therefore the LOSS invariants of $L_1$ and $L_2$ are  necessarily distinct as well). 

As we pointed out earlier, since we  only consider negative \edit{stabilizations}, contact $(-2)$-surgery on $L_i$ is the same as contact $(-1)$-surgery on $L_i^{-1}$, the once negatively-stabilized $L_i$ for $i =1,2$. Identifying  $\xi_{-2}(L_i)$ with $\xi_{-1}(L_i^{-1})$, in the following proof we denote  by $(Y_i, \xi_{-2}(L_i))$ the manifold with the associated contact structur
e obtained from contact \edit{$(-1)$-surgery} on $L_i^{-1}$. Smoothly   $Y_i \cong S^3_{-1}(E_n)$ for $i=1,2.$

\begin{proof} [Proof of Theorem \ref{-2 contact surgery}]
\begin{figure}[htb!]
\centering
\begin{tikzpicture}
\begin{scope}[thin, black!0!white]
          \draw  (-5, 0) -- (5, 0);
      \end{scope}
    \node at (-4,0){\includegraphics[scale=0.4]{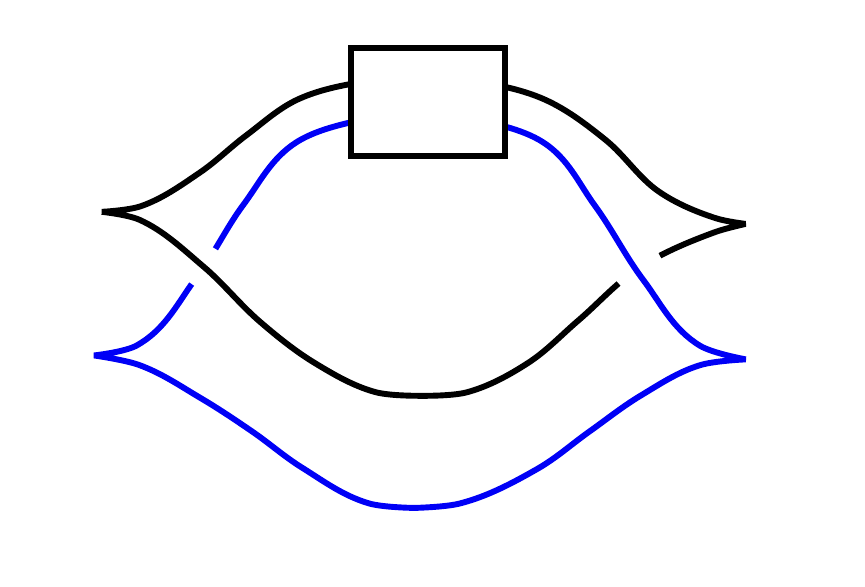}};
    \node at (0,1){\small $(-1)$-surgery};
    \node at (0,0.5){\small on $L_i^{-1}$};
    \node at (6,1){\small $-1$};
      \node at (-3.95,1.3){\small $L_i^{-1}$};
          \node at (4.05,1.3){\small $L_i^{-1}$};
        \draw  [very thick, -stealth]  (-1, 0) -- (1, 0);
         \node at (-1.75,-1)[text=blue]{\small $P_i$};
          \node at (6.1,-1)[text=blue]{\small $P'_i$};
       \node at (4,0){\includegraphics[scale=0.4]{Picture/thm41}};
\end{tikzpicture}
    \caption{Denoted by $L_i^{-1}$ is the once negatively-stabilized $L_i$ and $P_i$  its push-off. Then contact $(-1)$-surgery on $L_i^{-1}$ is the same as contact $(-2)$-surgery on $L_i$. \edit{Here $P_i'$ is the induced Legendrian of $P_i$ on contact manifold obtained by $(-1)$-surgery on $L_i^{-1}$ }}
    \label{fig:thm4_1}
\end{figure} 
Let $P_i$ be the Legendrian push-off of $L_i^{-1}$ in $S^3$, and denote by $P'_i$ the induced Legendrian knot in $Y_i$. See Figure \ref{fig:thm4_1}.
Apply Theorem \ref{Naturality of LOSS} to the pair $(P'_i,P''_i)$ in $(Y_i, \xi_{-2}(L_i))$, where $P''_i$ is induced by another Legendrian push-off of $L_i^{-1}$ in $(Y_i,\xi_{-2}(L_i))$. 
As contact $(-1)$-surgery on $L_i^{-1}$ and $(+1)$-surgery on $P_i''$ cancel,  we recover $(S^3, P_i)$ and
 obtain maps 
\begin{equation} \label{same smooth map for L_i}
    F_{-W_i,\mathfrak{s_i}}: \HFKm (-Y_i,P_i') \rightarrow \HFKm(-S^3,P_i) 
\end{equation}  with the property that $ F_{-W_i,\mathfrak{s}_i} (\mathfrak{L}(P_i'))=\mathfrak{L}(P_i)$, for $i=1,2$ respectively, where $W_i$ is the two-handle cobordism and $\spincs_i$ is given by Theorem \ref{Naturality of LOSS} \eqref{it: naturality2}.

\edit{Note that $L_1$, $L_2$, $P_1$, and $P_2$ are all  smoothly isotopic. Fixing an isotopy $\phi_{L_1 \to P_2}$ from $L_1$ to $P_2$, this induces an isomorphism $(\phi_{L_1 \to P_2})_*: \HFKm(-S^3,L_1) \to \HFKm(-S^3,P_2)$ by  \cite[Theorem 1.8]{JTZnmh}.  
Denote $\HFKm(-S^3,P_2)$ by  $\HFKm(-S^3,P)$. We identify $\HFKm(-S^3,L_1)$ with $\HFKm(-S^3,P)$ via the isomorphism $(\phi_{L_1 \to P_2})_*$, and consider the image of the LOSS invariants of $L_1$ under  this map.
The same process applies to for $L_2$ and $P_1$. It follows that we can
 view the LOSS invariants of $L_1$, $L_2$, $P_1$ and $P_2$ all as elements of $\HFKm(-S^3,P)$.} 

\edit{Similarly  we can identify $\HFKm(-Y,P_1')$ with $\HFKm(-Y,P_2')$, and again call the latter one $\HFKm(-Y,P')$. We can now also view  the LOSS invariants of $P_1'$ and $P_2'$ both as elements of $\HFKm(-Y,P')$.
Moreover, since  $P_1$ and $P_2$ are smoothly isotopic with the same rotation number, the induced cobordisms $W_1$ and $W_2$ are diffeomorphic and the diffeomorphism sends $\spincs_1$ to $\spincs_2$. Thus by Theorem 8.9 and Corollary 11.17 in \cite{Jcsf}, the naturality of canonical isomorphisms under cobordism maps allows us to define a single map }

\begin{equation} 
   F_{-W,\mathfrak{s}}: \HFKm (-{S^3_{-1}}(E_n),P') \rightarrow \HFKm(-S^3,P)
\end{equation}

with the property that $ F_{-W,\mathfrak{s}} (\mathfrak{L}(P_i'))=\mathfrak{L}(P_i),$ where  $(-W_i,\mathfrak{s}_i)= (-W,\mathfrak{s})$ for $ i=1,2.$  \edit{Since it was shown in \cite{OzsvathStipsiczContactsurgeryandtransverseinvariant}   that $\mathfrak{L}(L_1)$ and $\mathfrak{L}(L_2)$ are different when module mapping class group action,} we have  $\mathfrak{L}(L_1)\neq \mathfrak{L}(L_2)$ in $\HFKm(-S^3,P)$.  By Theorem \ref{LOSS invariant under negative stabilization}, LOSS invariant is unchanged under negative stabilization, therefore $\mathfrak{L}(P_1)\neq \mathfrak{L}(P_2)$ in $\HFKm(-S^3,P)$. It follows that $\mathfrak{L}(P_1')\neq \mathfrak{L}(P_2')$.

Now consider the map
\begin{equation}
    G: \HFKm (-{S^3_{-1}}(E_n),P') \rightarrow \HFa (-{S^3_{-1}}(E_n)).
\end{equation}
According to \cite[Proposition 2]{DGh} the Legendrian push-off $P'_i$ in the Legendrian surgery is smoothly isotopic to the dual knot of $L_i=E_n$
, and $-{S^3_{-1}}(E_n) \cong S^3_{+1}(-E_n).$ Thus the above $G$ map is the same as 
\begin{equation}
     G: \HFKm ({S^3_{+1}}(-E_n),\mu) \rightarrow \HFa ({S^3_{+1}}(-E_n)), 
\end{equation}
where $\mu$ is the meridian of the surgery knot $-E_n$. Using \cite[Lemma 6.6]{LOSS} it is easy to calculate that $\tb(P'_i)=0$ and $\rot(P'_i)=-1$, thus by \cite[Theorem 1.6]{OzsvathStipsiczContactsurgeryandtransverseinvariant} the Alexander grading of $\mathfrak{L}(P'_i)$ equals $1$ for all $i=1,2$. By Corollary \ref{cor: topinj}, $G$ is injective in the top Alexander grading $1$. Therefore by Lemma \ref{G maps LOSS to Contact invariant}, we conclude that $c(\xi_{-2}(L_1))\neq c(\xi_{-2}(L_2))$. 
\end{proof} 
\begin{remark}
    As discussed in Section \ref{subsubsec:stronginv}, the last step of the argument uses the naturality of Heegaard Floer homology \cite{JTZnmh}.
\end{remark}

\subsection{Case for $\pmb{r=-2-k}$, $\pmb{k\in \Z_{> 0}}$}
In this subsection we prove the following.

\begin{theorem} \label{-2-k contact surgery}
    Let $L_1$ and $L_2$ be two Legendrian representatives of $E_n$ with $n>3$ odd, that have same $tb=1$ and $\rot=0$ but different LOSS invariants. Then contact $(-2-k)$-surgery for $k\in \Z_{>0}$ on $L_1$ and $L_2$ gives contact manifolds that have different contact invariants.
\end{theorem}
The following well-known result (for example see \cite[Figure 9]{CasalsEtnyreKegelSteintraces}) will be essential in the later proofs.     \edit{For completeness}, we provide a proof for our setting.
\begin{lemma} \label{n+1/n surgery and n negative 
stabilization}
    The two Legendrian arcs $e_1$ and $e_2$ depicted in Figure \ref{New L} in the corresponding local contact surgery diagram are Legendrian isotopic. 
\end{lemma}
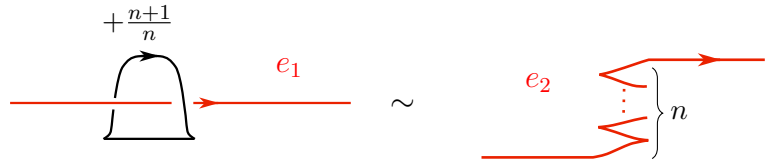
\begin{figure}[htb!]
\centering
\begin{tikzpicture}
\begin{scope}[thin, black!0!white]
          \draw  (-5, 0) -- (5, 0);
      \end{scope}
    \node at (-3,0){\includegraphics[scale=0.6]{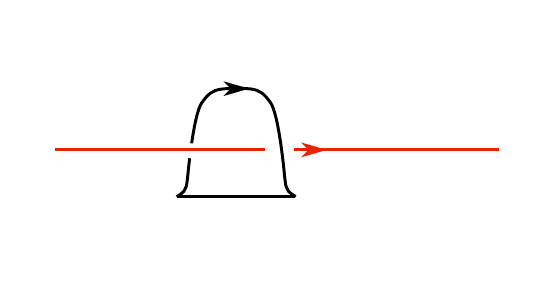}};
     \draw [decorate,decoration={brace,amplitude=4pt},xshift=0.5cm,yshift=0pt]
      (2.8,0.5) -- (2.8,-0.7) node [midway,right,xshift=0.1cm] {\small $n$};
      \node at (-3.5,1.1){\small $+\frac{n+1}{n}$};
        \node at (0,0){ $\sim$};
         \node at (-1.5,.5)[text=red]{ $e_1$};
          \node at (1.8,0.3)[text=red]{ $e_2$};
       \node at (3,0){\includegraphics[scale=0.7]{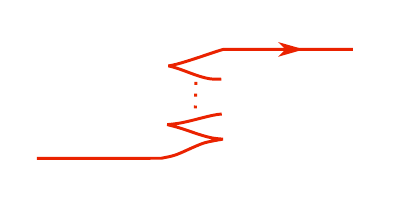}};
\end{tikzpicture}
    \caption{The number $n$ in the second diagram indicates that there are $n$ zigzags (i.e. $n$ \edit{negative stabilizations}). We call the surgery Legendrian in the first diagram the standard Legendrian meridian of $e_1$ (or more precisely, the knot which $e_1$ belongs to).}
    \label{New L}
\end{figure} 
\begin{figure}[htb!]
\centering
\begin{tikzpicture}
\begin{scope}[thin, black!0!white]
          \draw  (-5, -1) -- (5, 5);
      \end{scope}
    \node at (-4,4){\includegraphics[scale=0.6]{Picture/fig1a}};
      \node at (-3.5,4.6){\small $\frac{n+1}{n}$};
      \node at (5,4.6){\small $-1$};
           \node at (4.7,3.2){\small $+1$};
       \node at (4,4){\includegraphics[scale=0.6]{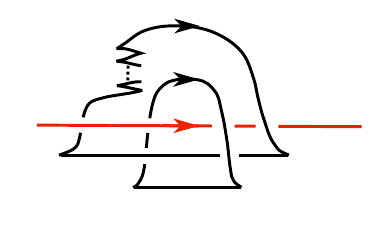}};
           \node at (-4,2.8){\textbf{a}};
         \node at (4,2.7){\textbf{b}};
              \node at (-2.8,1.2){\small $-1$};
               \node at (-3,-0.6){\small $+1$};
                     \node at (-4.87,0.64){\small $-n$};
        \node at (-4,0.5){\includegraphics[scale=0.6]{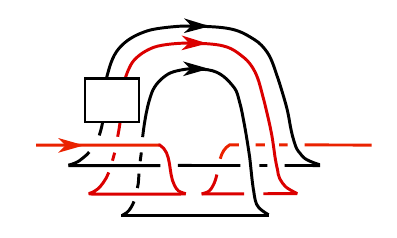}};
         \draw [decorate,decoration={brace,amplitude=4pt},xshift=0.5cm,yshift=0pt]
      (2.6,4.2) -- (2.6,4.8) node [midway,left,xshift=-0.05cm,yshift=0.1cm] {\small $n$};
           \node at (-4,-1.2){\textbf{c}};
         \node at (4,-1.2){\textbf{d}};
           \node at (4,0.5){\includegraphics[scale=0.7]{Picture/fig1b}};
          \draw [decorate,decoration={brace,amplitude=4pt},xshift=0.5cm,yshift=0pt]
      (3.1,0) -- (3.1,1) node [midway,left,xshift=-0.1cm] {\small $n$};
\end{tikzpicture}
    \caption{
      From \textbf{a} to \textbf{b}  we apply the DGS algorithm \cite{DGS} to   the contact $(\frac{n+1}{n})$-surgery, obtaining a sequence of contact $(+1)$- and $(-1)$-surgeries. Note that  all stabilizations are negative according to \edit{our convention}; from \textbf{b} to \textbf{c} we slide the red curve over the $(-1)$-framed $2$-handle using \cite[Proposition 1]{DGh}; from \textbf{c} to \textbf{d} \edit{we first replace the contact $(+1)$-framed $2$-handle by a Stein $1$-handle, cancel it with the $(-1)$-framed Stein $2$-handle, then apply Legendrian Reidemeister moves to the red curve.}}
     \label{New Lproof}
\end{figure} 
\begin{proof}
   The equivalence follows from \edit{the sequence of contact surgeries} and Stein handle moves shown  in Figure 
    \ref{New Lproof}. 
\end{proof}

Observe that using the above lemma we can infer that if $L$ is a Legendrian knot in some $(Y,\xi)$, then contact $(\frac{n+1}{n})$-surgery on the standard Legendrian meridian of $L$ send $L$ to $L^{-n}$. Therefore we can view Legendrian surgery on $L^{-n}$ as Legendrian surgery on $L$ followed by $(\frac{n+1}{n})$-surgery on the standard Legendrian meridian of $L^{-n}$. Now we are \edit{ready to prove Theorem \ref{-2-k contact surgery}.}

\begin{proof} [Proof of Theorem \ref{-2-k contact surgery}]
\begin{figure}[htb!]
\centering
\begin{tikzpicture}
\begin{scope}[thin, black!0!white]
          \draw  (-5, 0) -- (5, 0);
      \end{scope}
    \node at (-4,0){\includegraphics[scale=0.5]{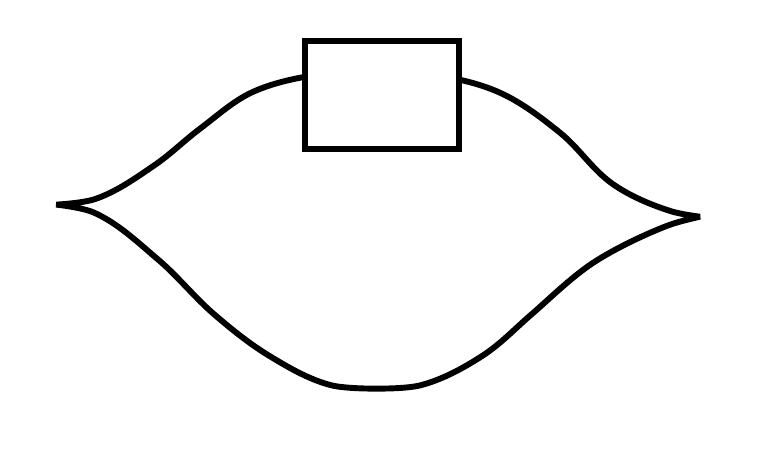}};
     \node at (-2,1){\small $-1$};
    \node at (6,1){\small $-1$};
      \node at (-3.97,1.1){ $L_i^{-1-k}$};
          \node at (4.05,1.1){ $L_i^{-1}$};
    \node at (0,0){\Large $\sim$};
     
          \node at (6.4,-0.5)[text=blue]{\small $\frac{k+1}{k}$};
       \node at (4,0){\includegraphics[scale=0.5]{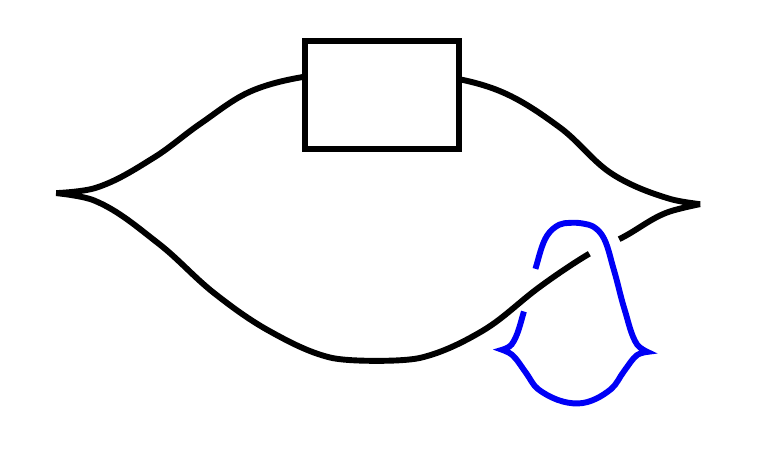}};
\end{tikzpicture}
    \caption{By lemma \ref{n+1/n surgery and n negative 
stabilization}, this two contact surgery diagram are equivalent.}
    \label{fig:thm4_2}
\end{figure}
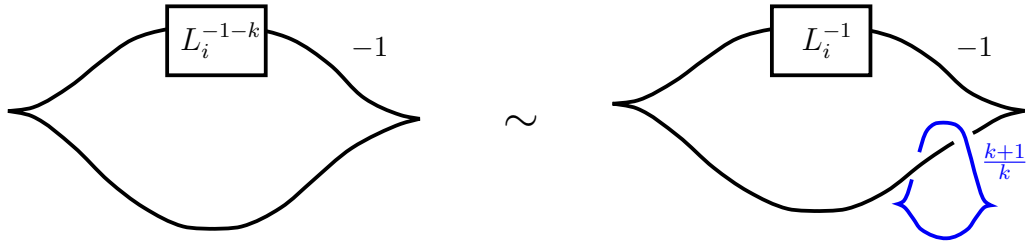 
To make the argument clear, we use the same notation as in the proof of Theorem \ref{-2 contact surgery}. We again let $P_i$ be the Legendrian push off of $L_i^{-1}$, $(Y_i,\xi_{-1}(L_i^{-1}))$ be the contact manifold obtained by Legendrian surgery on $L_i^{-1}$, $P_i'$ be \edit{the dual knot of $P_i$} in $(Y_i,\xi_{-1}(L_i^{-1}))$, and $(Y_i',\xi_{-2-k}(L_i))$ be the contact manifold obtained by  contact $(-2-k)$-surgery on $L_i$.

We first view contact $(-2-k)$-surgery on $L_i$ as Legendrian surgery on $L_i^{-k-1}$. Then by the observation above, Legendrian surgery on $L_i^{-k-1}$ is the same as Legendrian surgery on $L_i^{-1}$ followed by contact $(\frac{k+1}{k})$-surgery on the standard Legendrian meridian of $L_i^{-1}$. In other words we can view the contact manifold $(Y_i',\xi_{-2-k}(L_i))$ \edit{as obtained} by doing contact $(\frac{k+1}{k})$-surgery on the standard Legendrian meridian of $P_i$ in contact manifold $(Y_i,\xi_{-1}(L_i^{-1}))$. \edit{Moreover, again by} \cite[Proposition 2]{DGh}, this standard Legendrian meridian is Legendrian isotopic to $P_i'$ in $(Y_i,\xi_{-1}(L_i^{-1}))$. Recall that $\tb(P_i')=0$ and $\rot(P_i')=-1$, implies that smoothly we are also doing $(\frac{k+1}{k})$-surgery on $P_i'$. We also notice and denote $Y=Y_i=\edit{{S^3_{-1}}}(L_i)$  is a homology sphere, and $Y_i'=Y_{\frac{k+1}{k}}(P_i')$ is a rational homology sphere.

Thus by the \edit{Naturality Theorem} \ref{naturality of contact invariant} we obtain maps 
$$F_{-W_i,\mathfrak{s}_i}: \HFa(-Y \# L(k,-1)) \rightarrow \HFa(-(Y_{\frac{k+1}{k}}(P_i')))$$ with the \edit{property}  $$F_{-W_i,\mathfrak{s_i}}(c(\xi_{-2}(L_i))\otimes \tilde{c})= c(\xi_{-2-k}(L_i)).$$

Since the \edit{first Chern classes} are all torsion, we also have  $$\pm \langle c_1(\mathfrak{s}_i),[\Tilde{Z}] \rangle =p+(\rot(\edit{L_i})-\tb(\edit{L_i}))q-1=k+1+(-1-0)k-1=0,$$ which implies $\mathfrak{s}_1=\mathfrak{s}_2$. \edit{By the same argument as in the proof of Theorem \ref{-2 contact surgery} the  naturality under cobordism maps by \cite{Jcsf} and \cite{JTZnmh} implies that $F_{-W_i,\mathfrak{s}_i}, i=1,2$ above are again given by a single map} 

$$F_{-W,\mathfrak{s}}: \HFa(-Y \# L(k,-1)) \rightarrow \HFa(-(Y_{\frac{k+1}{k}}(P')))$$ with the property that for each $i=1,2$ $$F_{-W,\mathfrak{s}}(c(\xi_{-2}(L_i))\otimes \tilde{c})= c(\xi_{-2-k}(L_i)),$$ where $W=W_i$, $\mathfrak{s}=\mathfrak{s}_i$ and $P'$ is \edit{the underlying smooth knot} of $P_i'$. 

The goal now is to show this map $F_{-W,\mathfrak{s}}$ is injective. By Proposition \ref{mapping cone for contact invariant}, we only need to show that \begin{equation} \label{eq: inclusion of contact invariant into the mapping cone}
    (t,\hat{B}_s) \hookrightarrow \Xa_{-(k+1)/k}(-Y,P')
\end{equation}
is injective in homology, where $2t=(-1-0+1)k-2$, so $t=-1$.  (Proposition \ref{mapping cone for contact invariant} only determines   $\mathfrak{s}$ up  to conjugation. However here we have $\langle c_1(\mathfrak{s}),[\Tilde{Z}] \rangle =0$, 
so $\mathfrak{s}$ is self-conjugate.) 

The mapping cone corresponding to \edit{the} $\Spin^c$ structure \edit{$\mathfrak{s}_{-1}$} is given by 
\begin{equation} \label{eq: mappingcone}
\xymatrix@C=0.6in@R=0.6in{
& (-1,\hat{A}_{-1})\ar[dl]_{\hat{h}_{-1}}\ar[d]_{\hat{v}_{-1}} & (k,\hat{A}_{1})\ar[dl]_{\hat{h}_1} \ar[d]_{\hat{v}_1}\\
\cdots & (-1,\hat{B}_{-1}) & \cdots
}
\end{equation}
The genus of $P'$ is $1$ as it has the same knot complement as $E_n,$ so the map $\hat{v}_s$  (resp.~$\hat{h}_s$)  is an isomorphism for all $s>0$ (resp.~$s<0$).
By a standard truncation argument we see that
 $$(-1,\hat{B}_{-1}) \hookrightarrow (\Xa_{-(k+1)/k}(-Y, P'), -1)$$  in fact induces an isomorphism in homology. This gives us the desired conclusion $c(\xi_{-2-k}(L_1))\neq c(\xi_{-2-k}(L_2))$. \end{proof}

\subsection{Case for general rational $\pmb{r<0}$}

To obtain the general result we need to use Theorem \ref{DGS negative ration contact surgery}, thus for $r<0$ we write \begin{equation}
    r=- {\frac{x}{y}}=[a_1 + 1, a_2,\dots, a_\ell]=a_1+1-\cfrac{1}{a_2-\cfrac{1}{\dots-\cfrac{1}{a_\ell}}}
\end{equation}
where each $a_j\leq -2$. Aside from the DGS algorithm we also need the following proposition, which essentially states that the contact $(-1)$-surgery (Legendrian surgery) \edit{preserves} the distinction between contact invariants.
\begin{proposition} [\cite{Wantightcontactstrucutres} Theorem 1.1]
    {\label{Legendrian surgery preserve uniquness}}
Let $\xi^1$ and $\xi^2$ be two contact structures on a 3-manifold $Y$. Take any smooth knot $K$ in Y. Let $L_i$, $i=1, 2$ be a Legendrian representative of $K$ in $\xi^i$ such that \edit{Legendrian surgery on $L_i$ smoothly induce the same cobordism.}  Then 
    \begin{equation}
        c(\xi^1)\neq c(\xi^2) \text{ implies } c(\xi^1_{-1}(L_1))\neq c(\xi^2_{-1}(L_2)).
    \end{equation} 
\end{proposition} 
\edit{\begin{remark}
   The above Proposition is essentially \cite[Theorem 4.2]{OSHFandcontactstructures}, \cite[Lemma 2.11]{GhigginiOSinvariantsandfillabilityofcontactstructures} plus naturality. Also note that we do not require $K$ to be null-homologous as the naturality of contact invariants under the Legendrian surgeries or Stein cobordisms do not require $K$ to be null-homologous in \cite[Theorem 4.2]{OSHFandcontactstructures} or \cite[Lemma 2.11]{GhigginiOSinvariantsandfillabilityofcontactstructures}.
\end{remark}}
\begin{lemma}
   Let $r=[a_1, a_2,\dots, a_\l] <0$ where each $a_j \leq -2$ but not all are equal to $-2$. Let $L_1$ and $L_2$ be two Legendrian representatives of $E_n$ with $n>3$ odd, that have same $tb=1$ and $\rot=0$ but different LOSS invariants. Then  contact $r$-surgery on $L_1$ and $L_2$ gives contact manifolds that have different contact invariants.   
\end{lemma} \label{r not -1/k}
\begin{proof}
    The previous two subsections already prove the case when $\l=1$ so we assume $\l>1$. Suppose $a_t < -2 $ for some $t\in \{1,2,\edit{\dots},\l\}$. By \edit{the} DGS algorithm for $r<0$ (Theorem \ref{DGS negative ration contact surgery}), contact $r$-surgery on $L_i$ for $i=1,2$ is the same as a sequence of contact $(-1)$-surgeries on $L_{i}^{(j)}$ for $j=1,2,\edit{\dots},\ell$, where each $L_{i}^{(j)}$ is a \edit{Legendrian} push-off of $L_i$, with $| a_j + 2 | $ \edit{negative stabilizations}. See Figure \ref{fig:dgs1}.  If we  denote  $(Y_{i}^{(t)},\xi_{-1}(L_{i}^{(t)}))$ to be the contact $3$-manifold obtained by contact $(-1)$-surgery on $L_{i}^{(t)}$, we can view contact $r$-surgery on $L_i$ as obtained by first doing contact $(-1)$-surgery on $L_{i}^{(t)}$, then \edit{as a} contact $(-1)$-surgery in $(Y_{i}^{(t)},\xi_{-1}(L_{i}^{(t)}))$. Since $a_t<-2$, $L_{i}^{(t)}$ is \edit{negatively} stabilized at least once, thus by Theorem \ref{-2-k contact surgery} we have  $c(\xi_{-1}(L_{1}^{(t)}))\neq c(\xi_{-1}(L_{2}^{(t)}))$. Then by applying Proposition \ref{Legendrian surgery preserve uniquness} repeatedly on a sequence of contact $(-1)$-surgeries we infer that $c(\xi_r(L_1))\neq c(\xi_r(L_2)).$
\end{proof}
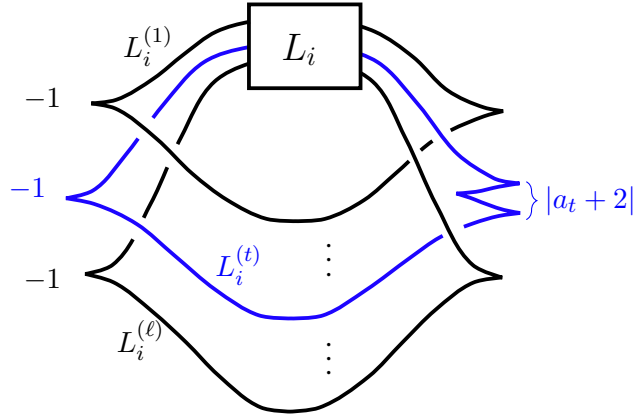
\begin{figure}[htb!]
\centering
\begin{tikzpicture}
    \node at (0,0){\includegraphics[scale=0.5]{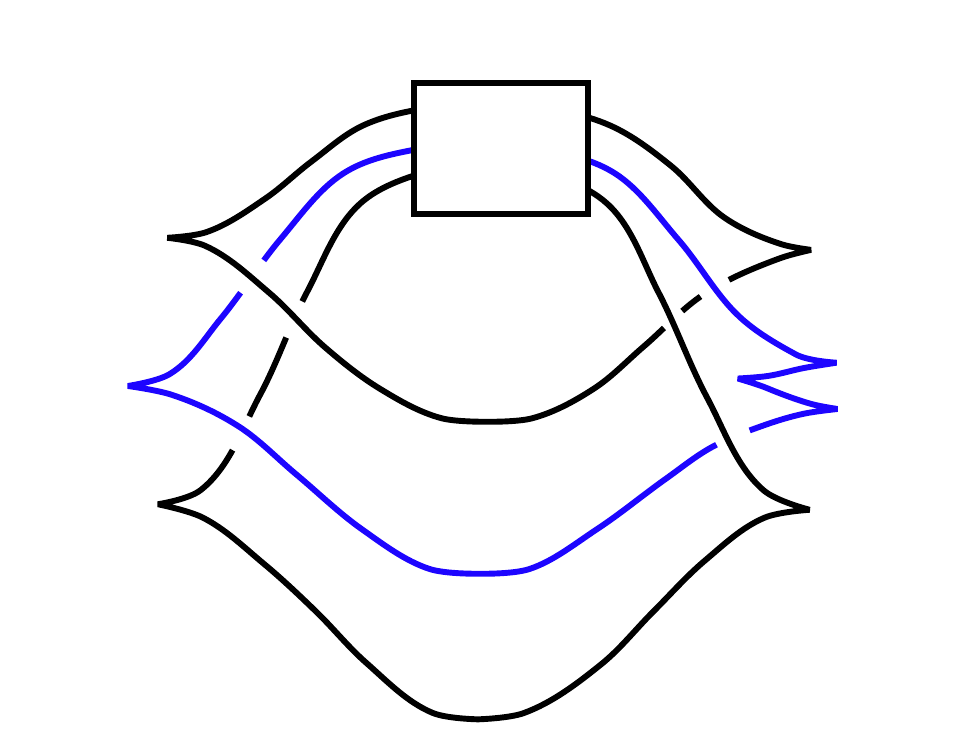}};
     \node at (0.1,1.85){\large $L_{i}$};
      \node at (-3.3,1.2){\small $-1$};
      \node at (-1.9,1.9){\small $L_{i}^{(1)}$};
       \node [text=blue] at (-3.5,0){\small $-1$};
        \node [text=blue] at (-0.7,-1){\small $L_{i}^{(t)}$};
       \node [rotate = 90] at (0.5,-0.9) {$\cdots$};
       \node at (-3.3,-1.2){\small $-1$};
        \node [rotate = 90] at (0.5,-2.2) {$\cdots$};
         \node at (-2,-2){\small $L_{i}^{(\l)}$};
          \draw [blue,decorate,decoration={brace,amplitude=4pt},xshift=0.5cm,yshift=0pt]
      (2.6,0.1) -- (2.6,-0.4) node [midway,right,xshift=0.1 cm] [text=blue] {\small $| a_t + 2 |$};
\end{tikzpicture}
    \caption{The contact $r$-surgery on $L_i$ for $i=1,2$ according to \edit{the} DGS algorithm. \edit{The} surgery coefficient is for contact surgery. \edit{For each $j\in \{1,2,\edit{\dots},\l\},$ each component $L_{i}^{(j)}$ is negatively stabilized $| a_j + 2 | $ times. In particular, $L_i^{(t)}$ denotes a component indexed by $t$ with $a_t<-2$.}}
    \label{fig:dgs1}
\end{figure} 

Now we are left with one case where $r=[a_1+1,\dots,a_\l]$ and all $a_j=-2$. This will be included in the proof of Theorem \ref{negative rational contact surgery on twist knots}. 

\begin{proof} [Proof of Theorem \ref{negative rational contact surgery on twist knots}]
First we note that $a_j=-2$ for all $j=1,2,\dots,\l$ if and only if $r=-\frac{1}{\l}$, thus
    as we discussed above, the only remaining case is when $r=-\frac{1}{\l}$. We let $r'=r-1=-\frac{1}{\l}-1$ and we denote $(Y,\xi_{r}(L_i))$ and $(Y',\xi_{r'}(L_i))$ to be the contact \edit{$3$-manifolds} obtained by taking contact $r=-\frac{1}{\l}$ and $r'=-\frac{1}{\l}-1$ surgeries on $L_i$ respectively. Then by applying Lemma \ref{n+1/n surgery and n negative 
stabilization} to the standard Legendrian meridian of $L_{i}$, we can view $(Y',\xi_{r'}(L_i))$ as obtained from contact $(+2)$-surgery on the standard Legendrian meridian of $L_{i}$ in $(Y,\xi_r(L_i))$. Moreover, when $\l>1$, both $Y$ and $Y'$ are rational homology spheres $(tb(L_i)=1)$ thus by (1) in Theorem \ref{naturality of contact invariant} \edit{(where the contact surgery coefficient is $+2$)} we obtain maps 

$$F_{-W_i,\mathfrak{s_i}}: \HFa(-Y) \rightarrow \HFa(-Y')$$ with the \edit{property} 

$$F_{-W_i,\mathfrak{s_i}}(c(\xi_r(L_i)))= c(\xi_{r'}(L_i)).$$ 
$L_i$ have the same smooth knot type, $\tb$ and $\rot$ implies their corresponding standard Legendrian meridians have the same classical invariants (even if they are just rationally null-homologous). Thus \edit{(by Theorem \ref{thm: same Spin^c for Legendrian with classical invariants})} $\spincs_1=\spincs_2$ which implies \edit{that the} two maps $F_{-W_i,\mathfrak{s_i}}$ are equivalent . Notice that $r'$ is not of the form  $-\frac{1}{k}$, thus by Theorem \ref{r not -1/k} we have $c(\xi_{r'}(L_1)) \neq c(\xi_{r'}(L_2))$ which implies $c(\xi_r(L_1)) \neq c(\xi_r(L_2))$.
\end{proof}

\begin{figure}[htb!]
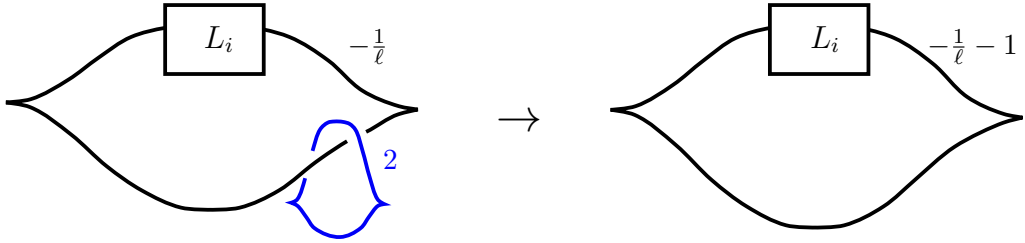

\centering
\begin{tikzpicture} 
\begin{scope}[thin, black!0!white]
          \draw  (-5, 0) -- (5, 0);
      \end{scope}
    \node at (-4,0){\includegraphics[scale=0.5]{Picture/newthm42b}};
     \node at (-2,1){\small $-\frac{1}{\l}$};
      \node at (-1.7,-0.5)[text=blue]{\small $2$};
    \node at (6,1){\small $-\frac{1}{\l}-1$};
      \node at (-3.97,1.1){ $L_i$};
          \node at (4.05,1.1){ $L_i$};
    \node at (0,0){\Large $\rightarrow$};
       \node at (4,0){\includegraphics[scale=0.5]{Picture/thm42a}};
\end{tikzpicture}
\caption{contact $(-\frac{1}{\l}-1)$-surgery on $L_i$ is obtained by contact $+2$-surgery on the standard Legendrian meridian of  $(-\frac{1}{\l})$-surgery on $L_i$ }
\end{figure} 

\begin{remark}

Although  not explicitly written in the paper \cite{MarkBulentNaturalityofContactinvariant}, it has been pointed out by Baldwin \cite{BaldwinPersonal} that for the proof of \cite[Proposition 2.3]{BaldwinCappingoff} the \edit{3-manifolds} need to satisfies certain condition to make the naturality of contact invariant (Theorem \ref{naturality of contact invariant}) work, and being a rational homology sphere is a sufficient conditions.
   
This is the reason why we are not able to obtain different contact invariants on Legendrian surgery of $L_i$, because when \edit{it is} not a rational homology sphere the condition is not necessarily hold, and in our cases the naturality result  \edit{does not apply} due to the above proof and the fact that $c(\xi_{-1}(L_1))=c(\xi_{-1}(L_2))$ \cite{LiscaStipsiczNotesoncontactinvariants}.   
\end{remark}

\section{\edit{Results} for Legendrian knot $E(m,n)$} \label{sec: Emn}

In this section we prove  Theorem \ref{negative rational contact surgery on two bridge knots}. Recall from \cite[Theorem 5.2]{WanNaturalityofLOSSinvariant} \edit{that  the} Legendrian representatives $L_1$ and $L_2$  of $E(m,n)$ with different LOSS invariants are obtained by adding $m-1$ half positive-twist to the Legendrian representatives $L_1'$ and $L_2'$ of $E_n$, where $L_1'$ and $L_2'$ have different LOSS invariants (see the top left of Figure \ref{fig:obtainEmn} \edit{when $m=3$}). 



\begin{proof} [Proof of Theorem \ref{negative rational contact surgery on two bridge knots}]
\begin{figure}[htb!]
\centering
  \begin{minipage}{\linewidth}
\centering
\subfloat[]{
\begin{tikzpicture}
\begin{scope}[thin, black!0!white]
          \draw  (-5, 0) -- (5, 0);
      \end{scope}
    \node at (-4,0){\includegraphics[scale=0.35]{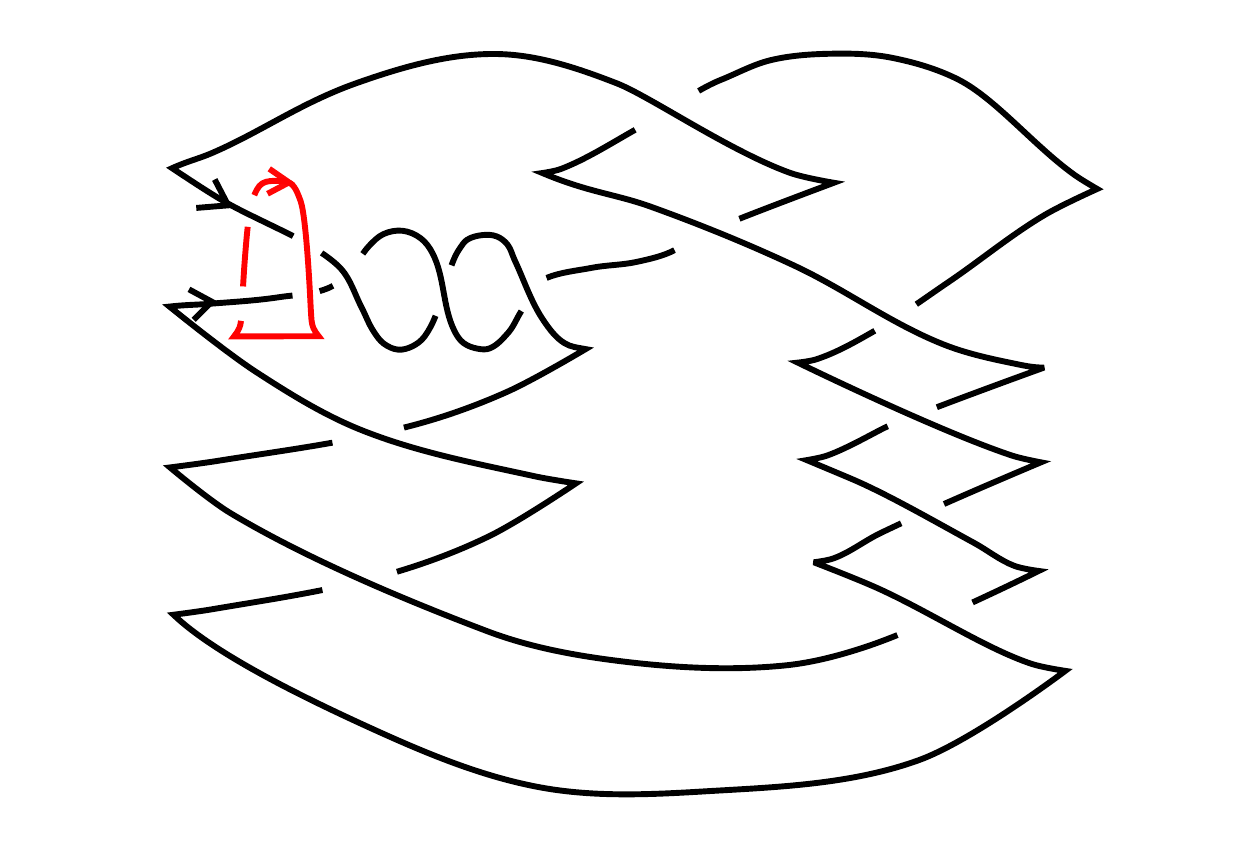}};
      \node at (-3.97,2.5){ $L_i$};
          \node at (4.05,2.5){ $L'_i$};
           \node at (-5.55,1.5) [color=red]{\small $+2$};
    \node at (0,0){\Large $\sim$};
       \node at (4,0){\includegraphics[scale=0.35]{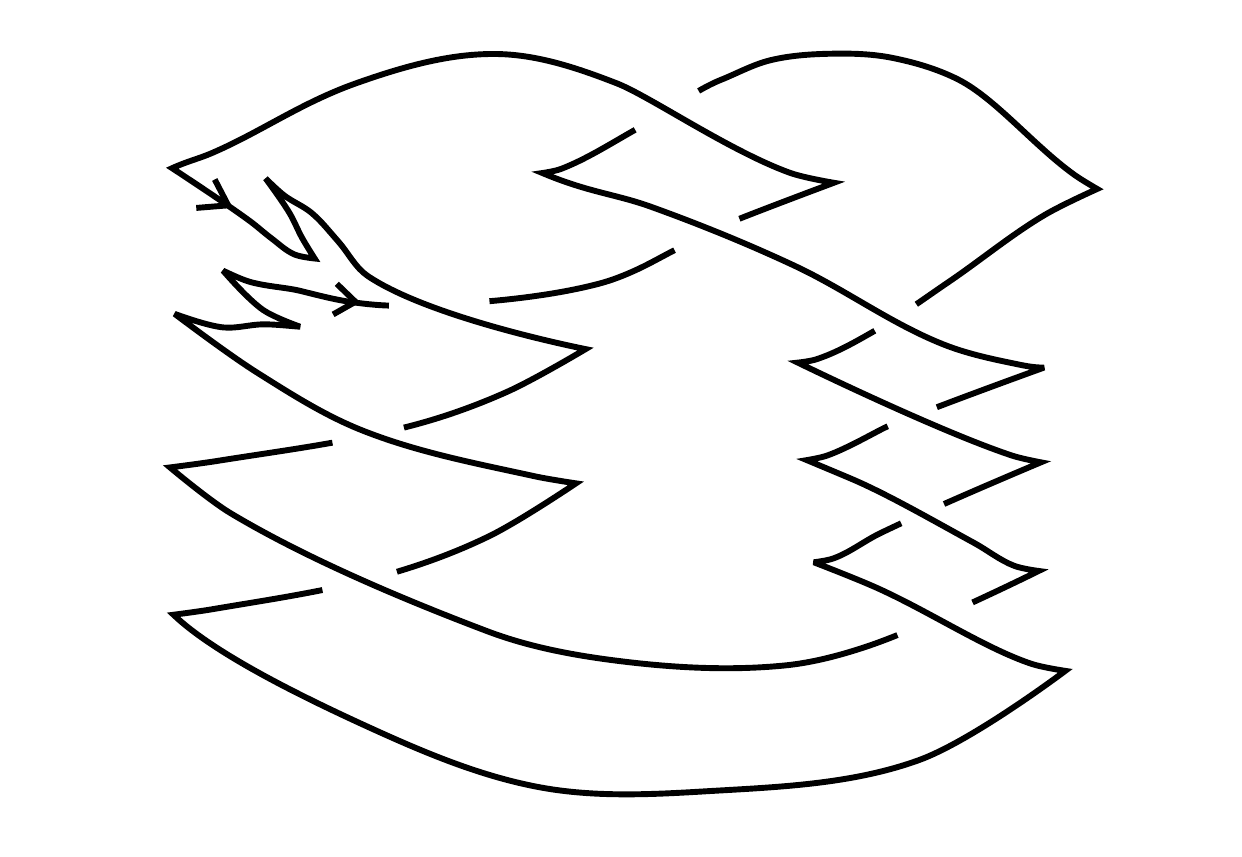}};
       \end{tikzpicture}
}
\end{minipage}\\
  \begin{minipage}{\linewidth}
\centering
\subfloat[]{
\begin{tikzpicture}
       \node at (-5,0){\includegraphics[scale=0.4]{Picture/thm42a}};
   \node at (-1,1){\includegraphics[scale=0.8]{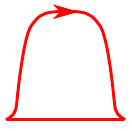}};
      \node at (-3.75,1){\small $r$};
    \node at (6.75,1){\small $r-m+1$};
    \node at (0,1)[color=red] {\small $+2$};
      \node at (-4.95,0.9){\small $L_i$};
          \node at (5,0.9){\small $L'_i$};
              \node at (1.75,0){ $\cdots \cdots$};
        \draw  [very thick, -stealth]  (-2.5, 0) -- (0.5, 0);
       \node at (5,0){\includegraphics[scale=0.4]{Picture/thm42a}};
        \draw [decorate,thick,decoration={brace,amplitude=4pt},xshift=0cm,yshift=0pt]
      (2.5,-0.5) -- (-2.5,-0.5) node [midway, yshift= -0.5 cm]  {repeat $(m-1)/2$ times};
\end{tikzpicture}
}
\end{minipage}
    \caption{Diagram for the proof of Theorem \ref{negative rational contact surgery on two bridge knots}. Top row shows the case of one Legendrian representative of $E(m,n)$ when $m=3$.  }
    \label{fig:obtainEmn}
\end{figure} 
    First fix a negative rational contact surgery coefficient $r$, and assume $m>1$ and odd. Denote $(Y,\xi_r(L_i))$ to be the contact $3$-manifold obtained by taking contact $r$-surgery on $L_i$, and $(Y',\xi_{r-m+1}(L_i'))$ to be the contact $3$-manifold obtained by taking contact $(r-m+1)$-surgery on $L_i'$ (recall that contact $(r-m+1)$-surgery on $L_i'$ is the same as contact $r$-surgery on $(m-1)$ negatively stabilized $L_i'$).

    As we described above, by undoing the full twist using contact $(+2)$-surgery on a sequence of standard Legendrian unknots, we will bring the Legendrian knot $L_i$ back to $L_i'$ with $m-1$ negative stabilization. Thus after applying contact $r$-surgery on $L_i$ we have the following maps

    $$F_{-W_i,\mathfrak{s_i}}: \HFa(-Y) \rightarrow \HFa(-Y'). $$ 

    Moreover since $\tb(L_i)=m$ for both $i=1,2$, when $r\neq -m$, both $Y$ and $Y'$ are rational homology spheres.  Therefore the naturality of contact invariant, i.e. Theorem \ref{naturality of contact invariant} applies and we get 
    $$F_{-W_i,\mathfrak{s_i}}(c(\xi_r(L_i)))= c(\xi_{r-m+1}(L_i')).$$

    \edit{Again, since both $L_i$ share the same smooth knot type, $\tb$ and $\rot$, by the same argument as in the proof of Theorem \ref{-2 contact surgery} the  naturality under cobordism maps by \cite{Jcsf} and \cite{JTZnmh} implies that $F_{-W_i,\mathfrak{s}_i}, i=1,2$ above are given by a single map} $$F_{-W,\mathfrak{s}}: \HFa(-Y) \rightarrow \HFa(-Y'). $$  Since $r-m+1\neq 1$, by Theorem \ref{negative rational contact surgery on twist knots} we have $c(\xi_{r-m+1}(L_1')) \neq c(\xi_{r-m+1}(L_2'))$. Combining the above, we infer  $c(\xi_{r}(L_1)) \neq c(\xi_{r}(L_2))$.

\end{proof}

\begin{remark}
    An alternative approach is by following the proof outline in Section \ref{sec:proof}, replacing $E_n$ by $E(m,n)$ at each instance. This method produces an analogous result to Theorem \ref{-2 contact surgery}, but fails at the second step, when extending to $(-2-k)$-surgery. The main reason is that the genus of $E(m,n)$ is greater than $1,$ and increases when $m$ increase, so the truncation of the mapping cone \eqref{eq: mappingcone} becomes more complicated. Thus we are unable to conclude the injectivity which is necessary for the arguments.
\end{remark}

\bibliographystyle{amsalpha}
\bibliography{bibliography}

\end{document}